\documentclass[11pt,reqno]{amsart}
\usepackage{amssymb,latexsym}
\usepackage{graphicx}

\usepackage{tikz}
\usepackage{circuitikz}
\usetikzlibrary{arrows}
\usepackage{url}		
\calclayout

\makeatletter
\newcommand{\monthyear}[1]{%
  \def\@monthyear{\uppercase{#1}}}
\newcommand{\volnumber}[1]{%
  \def\@volnumber{\uppercase{#1}}}
\AtBeginDocument{%
\def\ps@plain{\ps@empty
  \def\@oddfoot{\@monthyear \hfil \thepage}%
  \def\@evenfoot{\thepage \hfil \@volnumber}}
\def\ps@firstpage{\ps@plain}
\def\ps@headings{\ps@empty
  \def\@evenhead{%
    \setTrue{runhead}%
    \def\thanks{\protect\thanks@warning}%
    \uppercase{The Fibonacci Quarterly}\hfil}%
  \def\@oddhead{%
    \setTrue{runhead}%
    \def\thanks{\protect\thanks@warning}%
    \hfill\uppercase{An Infinite 2-Dimensional Array}}%
  \let\@mkboth\markboth
  \def\@evenfoot{%
    \thepage \hfil \@volnumber}%
  \def\@oddfoot{%
    \@monthyear \hfil \thepage}%
  }%
\footskip=25pt
\pagestyle{headings}%
}
\makeatother

\theoremstyle{plain}
\numberwithin{equation}{section}
\newtheorem{thm}{Theorem}[section]
\newtheorem{theorem}[thm]{Theorem}
\newtheorem{lemma}[thm]{Lemma}
\newtheorem{example}[thm]{Example}
\newtheorem{definition}[thm]{Definition}

\newtheorem{comment}[thm]{Comment}
\newtheorem{conjecture}[thm]{Conjecture}
\newtheorem{corollary}[thm]{Corollary}

\begin{document}
\monthyear{Month Year}
\volnumber{Volume, Number}
\setcounter{page}{1}

\title{An Infinite 2-Dimensional Array Associated With Electric Circuits}
\author{Emily Evans}
\address{Brigham Young University}
\email{EJEvans@math.byu.edu}
\author{Russell Jay Hendel}
\address{Towson University}
\email{RHendel@Towson.Edu}

\begin{abstract}
Except for Koshy who devotes seven pages to applications of Fibonacci Numbers to electric circuits, most books and the Fibonacci Quarterly have been relatively silent on applications of graphs and electric circuits to Fibonacci numbers. This paper continues a recent trend of papers studying the interplay of graphs, circuits, and Fibonacci numbers by presenting and studying the Circuit Array, an infinite 2-dimensional array whose entries are electric resistances labelling edge values of circuits associated with a family of graphs. The Circuit Array has several features distinguishing it from other more familiar arrays such as the Binomial Array and Wythoff Array. For example, it can be proven modulo a strongly supported conjecture that the numerators of its left-most diagonal do not satisfy any linear, homogeneous, recursion, with constant coefficients (LHRCC).  However, we conjecture with supporting numerical evidence an asymptotic formula involving $\pi$ satisfied by the left-most diagonal of the Circuit Array.
\end{abstract}

\maketitle

\section{Electrical Circuits, Linear 2-trees, and Fibonacci Numbers}\label{sec:s1_circuits}

Koshy \cite[pp. 43-49]{Koshy} lists applications of electrical circuits yielding interesting Fibonacci identities.  However, aside from this, most books and as well as the issues of the Fibonacci Quarterly have been mostly silent on this application.

To begin our review of the recent literature, which has renewed interest in this application, first, recall one modern graph metric, effective resistance, requires that the graph be represented as an electric circuit with edges in the graph represented by resistors. Figure \ref{fig:pawcircuit} illustrates this.

\begin{figure}
    \centering
    \begin{tabular}{||c| c||}
    \hline
\begin{tikzpicture}
\draw [line width=1pt,color=black] (0,0)--  (0,3)--(2.60,1.5) -- (0,0);
\draw [line width=1pt,color=black]   (2.60,1.5) -- (5.60,1.5);
\coordinate (A)  at (0,0); 
\coordinate (B)   at (0,3);
\coordinate (C)   at (2.60,1.5);
\coordinate (D)   at (5.60,1.5);

\node at (A) [left] {A};
\node at (B) [left] {B};
\node at (C) [above] {C};
\node at (D) [above] {D};
\draw [fill=black] (0,3.) circle (2pt);
\draw [fill=black] (0,0) circle (2pt);
\draw [fill=black] (2.60,1.5) circle (2pt);
\draw [fill=black] (5.60,1.5) circle (2pt);
\node [below] at (2.8,-.7) {Panel A - Paw graph};;

\end{tikzpicture}&
\begin{circuitikz}
\draw (0,0) to [R, *-*] (0,3); 
\draw (0,3) to [R, *-*] (2.60,1.5); 
\draw (2.60,1.5) to [R, *-*] (0,0); 
\draw (2.60,1.5) to [R, *-*] (5.60,1.5); 
\coordinate (A)  at (0,0); 
\coordinate (B)   at (0,3);
\coordinate (C)   at (2.60,1.5);
\coordinate (D)   at (5.60,1.5);

\node at (A) [left] {A};
\node at (B) [left] {B};
\node at (C) [above] {C};
\node at (D) [above] {D};

\node [below] at (2.8,-.7) {Panel B - Paw graph as a circuit};;

\end{circuitikz}\\
\hline
\end{tabular}
\caption{Illustration of a graph and its associated circuit. }\label{fig:pawcircuit}
\end{figure}
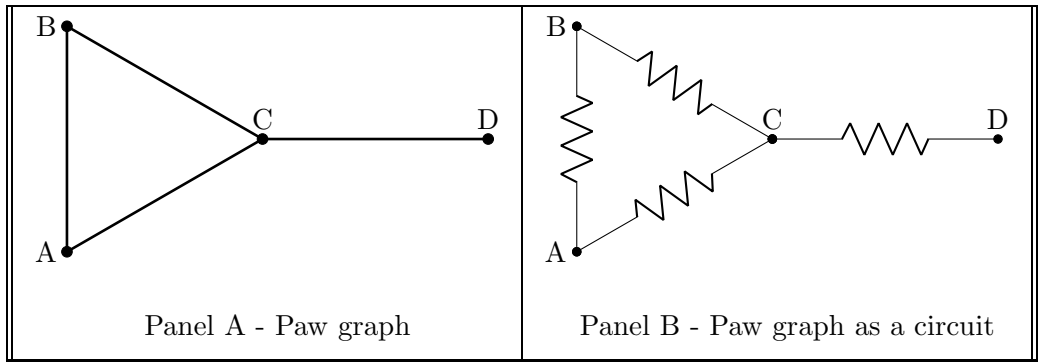

Several papers \cite{Barrett9, Barrett0} have explored effective resistances in electrical circuits whose underlying graphs are so-called linear 2-trees. In addition to showing that these effective resistances are rational functions of Fibonacci numbers, these circuits naturally give rise to interesting and new Fibonacci identities. For example the identities

\begin{equation}\label{eq:fibid1}
    \sum_{i = 1}^{m} \frac{F_i F_{i+1}}{L_i L_{i+1}} = \frac{(m+1) L_{m+1} - F_{m+1}}{5 L_{m+1}}, \quad\text{for $m \geq 1$,}
\end{equation}
and
for $k=3, 4, \dots, n-2$,
\begin{equation}\label{eq:wayne1a}
\sum_{j=3}^k {[(-1)^j F_{n-2j+1}(F_{n}+F_{j-2}F_{n-j-1})]}=-F_{k-2}F_{k+1}F_{n-k-2}F_{n+1-k}.
\end{equation}

To appreciate these recent contributions we provide additional background.
Effective resistance, also termed resistance distance in the literature, is a graph metric whose definition was motivated by the consideration of a graph as an electrical circuit.  More formally, given a graph, we determine the effective resistance between any two vertices in that graph by assuming that the graph represents an electrical circuit with resistances on each edge. Given any two vertices labeled $i$ and $j$ for convenience assume that one unit of current flows into vertex $i$ and one unit of current flows out of vertex $j$.  The potential difference $v_i - v_j$ between nodes $i$ and $j$ needed to maintain this current is the {\it effective resistance} between $i$ and $j$. Figure \ref{fig:pawcircuit} illustrates this.

Recent prior works~\cite{Barrett9, Barrett0, Barrett0b}, study effective resistance in a class of graphs termed {\it linear 2-trees}, also known as 2-paths, which we now define and illustrate. 
\begin{definition}\label{def:2tree}
In graph--theoretic language, a 2-tree is defined inductively as follows
\begin{enumerate}
    \item $K_3$ is a 2-tree.
    \item If $G$ is a 2-tree, the graph obtained by inserting a vertex adjacent to the two vertices of an edge of $G$ is a 2-tree.
\end{enumerate}
 A linear $2$-tree (or $2$-path) is a $2$-tree in which exactly two vertices have degree $2$.  For an illustration of two sample linear 2--trees see Figure~\ref{fig:2tree}.
\end{definition}
 
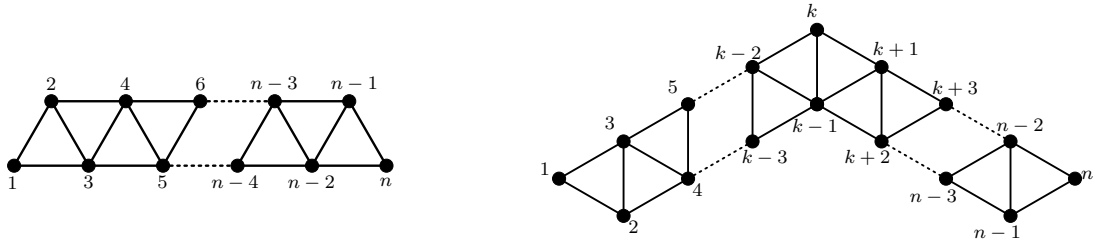
\begin{figure}[ht!]
\begin{center}
\scalebox{.9}{
\begin{tikzpicture}[line cap=round,line join=round,>=triangle 45,x=1.0cm,y=1.0cm,scale = 1.1]
\draw [line width=1.pt] (-3.,0.)-- (-2.,0.);
\draw [line width=1.pt] (-2.,0.)-- (-1.,0.);
\draw [line width=1.pt,dotted] (-1.,0.)-- (0.,0.);
\draw [line width=1.pt] (0.,0.)-- (1.,0.);
\draw [line width=1.pt] (1.,0.)-- (2.,0.);
\draw [line width=1.pt] (2.,0.)-- (1.5,0.866025403784435);
\draw [line width=1.pt] (1.5,0.866025403784435)-- (1.,0.);
\draw [line width=1.pt] (1.,0.)-- (0.5,0.8660254037844366);
\draw [line width=1.pt] (0.5,0.8660254037844366)-- (0.,0.);
\draw [line width=1.pt] (-0.5,0.8660254037844378)-- (-1.,0.);
\draw [line width=1.pt] (-1.,0.)-- (-1.5,0.8660254037844385);
\draw [line width=1.pt] (-1.5,0.8660254037844385)-- (-2.,0.);
\draw [line width=1.pt] (-2.,0.)-- (-2.5,0.8660254037844388);
\draw [line width=1.pt] (-2.5,0.8660254037844388)-- (-3.,0.);
\draw [line width=1.pt] (-2.5,0.8660254037844388)-- (-1.5,0.8660254037844385);
\draw [line width=1.pt] (-1.5,0.8660254037844385)-- (-0.5,0.8660254037844378);
\draw [line width=1.pt,dotted] (-0.5,0.8660254037844378)-- (0.5,0.8660254037844366);
\draw [line width=1.pt] (0.5,0.8660254037844366)-- (1.5,0.866025403784435);
\begin{scriptsize}
\draw [fill=black] (-3.,0.) circle (2.5pt);
\draw[color=black] (-3.02279181666165,-0.22431183338253265) node {$1$};
\draw [fill=black] (-2.,0.) circle (2.5pt);
\draw[color=black] (-2.0001954862580344,-0.22395896857501957) node {$3$};
\draw [fill=black] (-2.5,0.8660254037844388) circle (2.5pt);
\draw[color=black] (-2.5018465162673555,1.100526386601008717) node {$2$};
\draw [fill=black] (-1.5,0.8660254037844385) circle (2.5pt);
\draw[color=black] (-1.5081915914412003,1.100526386601008717) node {$4$};
\draw [fill=black] (-1.,0.) circle (2.5pt);
\draw[color=black] (-1.0065405614318794,-0.22290037415248035) node {$5$};
\draw [fill=black] (-0.5,0.8660254037844378) circle (2.5pt);
\draw[color=black] (-0.4952423962300715,1.100526386601008717) node {$6$};
\draw [fill=black] (0.,0.) circle (2.5pt);
\draw[color=black] (-0.03217990699069834,-0.22431183338253265) node {$n-4$};
\draw [fill=black] (0.5,0.8660254037844366) circle (2.5pt);
\draw[color=black] (0.4887653934035965,1.103344389715898) node {$n-3$};
\draw [fill=black] (1.,0.) circle (2.5pt);
\draw[color=black] (0.9904164234129174,-0.22431183338253265) node {$n-2$};
\draw [fill=black] (1.5,0.866025403784435) circle (2.5pt);
\draw[color=black] (1.5692445349621338,1.1042991524908385) node {$n-1$};
\draw [fill=black] (2.,0.) circle (2.5pt);
\draw[color=black] (1.993718483431559,-0.22431183338253265) node {$n$};
\draw[color=black] (1.993718483431559,-1) node {};
\end{scriptsize}
\end{tikzpicture}}
\qquad \qquad 
\scalebox{.9}{
\begin{tikzpicture}[line cap=round,line join=round,>=triangle 45,x=1.0cm,y=1.0cm,scale = 1.1]
\draw [line width=.8pt] (-5.464101615137757,2.)-- (-4.598076211353318,1.5);
\draw [line width=.8pt] (-4.598076211353318,1.5)-- (-4.598076211353318,2.5);
\draw [line width=.8pt] (-4.598076211353318,2.5)-- (-3.732050807568879,2.);
\draw [line width=.8pt] (-3.732050807568879,2.)-- (-3.7320508075688785,3.);
\draw [line width=.8pt] (-2.8660254037844393,2.5)-- (-2.866025403784439,3.5);
\draw [line width=.8pt] (-2.866025403784439,3.5)-- (-2.,3.);
\draw [line width=.8pt] (-2.,3.)-- (-2.,4.);
\draw [line width=.8pt] (-2.,4.)-- (-1.1339745962155612,3.5);
\draw [line width=.8pt] (-1.1339745962155612,3.5)-- (-2.,3.);
\draw [line width=.8pt] (-2.,3.)-- (-1.1339745962155616,2.5);
\draw [line width=.8pt] (-1.1339745962155616,2.5)-- (-1.1339745962155612,3.5);
\draw [line width=.8pt] (-1.1339745962155612,3.5)-- (-0.2679491924311225,3.);
\draw [line width=.8pt] (-0.2679491924311225,3.)-- (-1.1339745962155616,2.5);
\draw [line width=.8pt,dotted] (-1.1339745962155616,2.5)-- (-0.26794919243112303,2.);
\draw [line width=.8pt,dotted] (-0.2679491924311225,3.)-- (0.5980762113533165,2.5);
\draw [line width=.8pt] (0.5980762113533165,2.5)-- (-0.26794919243112303,2.);
\draw [line width=.8pt] (-0.26794919243112303,2.)-- (0.5980762113533152,1.5);
\draw [line width=.8pt] (0.5980762113533152,1.5)-- (0.5980762113533165,2.5);
\draw [line width=.8pt] (0.5980762113533165,2.5)-- (1.464101615137755,2.);
\draw [line width=.8pt] (1.464101615137755,2.)-- (0.5980762113533152,1.5);
\draw [line width=.8pt] (-2.,4.)-- (-2.866025403784439,3.5);
\draw [line width=.8pt,dotted] (-2.866025403784439,3.5)-- (-3.7320508075688785,3.);
\draw [line width=.8pt] (-3.7320508075688785,3.)-- (-4.598076211353318,2.5);
\draw [line width=.8pt] (-4.598076211353318,2.5)-- (-5.464101615137757,2.);
\draw [line width=.8pt] (-4.598076211353318,1.5)-- (-3.732050807568879,2.);
\draw [line width=.8pt,dotted] (-3.732050807568879,2.)-- (-2.8660254037844393,2.5);
\draw [line width=.8pt] (-2.8660254037844393,2.5)-- (-2.,3.);
\begin{scriptsize}
\draw [fill=black] (-2.,4.) circle (2.5pt);
\draw[color=black] (-2.06648828953202,4.259331085745072) node {$k$};
\draw [fill=black] (-1.1339745962155612,3.5) circle (2.5pt);
\draw[color=black] (-0.9407359844212153,3.7428094398707032) node {$k+1$};
\draw [fill=black] (-2.,3.) circle (2.5pt);
\draw[color=black] (-2.0267558552339913,2.6989231016718021) node {$k-1$};
\draw [fill=black] (-2.866025403784439,3.5) circle (2.5pt);
\draw[color=black] (-3.0597991469827295,3.689832860806665) node {$k-2$};
\draw [fill=black] (-2.8660254037844393,2.5) circle (2.5pt);
\draw[color=black] (-2.702207238300474,2.2919066737377372) node {$k-3$};
\draw [fill=black] (-1.1339745962155616,2.5) circle (2.5pt);
\draw[color=black] (-1.3248161826354898,2.2919066737377372) node {$k+2$};
\draw [fill=black] (-2.,3.) circle (2.5pt);
\draw [fill=black] (-0.2679491924311225,3.) circle (2.5pt);
\draw[color=black] (-0.14608729846064736,3.213043649230324) node {$k+3$};
\draw [fill=black] (-3.7320508075688785,3.) circle (2.5pt);
\draw[color=black] (-3.9339127015393545,3.2395319387623434) node {$5$};
\draw [fill=black] (-3.732050807568879,2.) circle (2.5pt);
\draw[color=black] (-3.6028090823891175,1.8488967383313488) node {$4$};
\draw [fill=black] (-0.26794919243112303,2.) circle (2.5pt);
\draw[color=black] (-0.4374584833128556,1.7488967383313488) node {$n-3$};
\draw [fill=black] (-4.598076211353318,2.5) circle (2.5pt);
\draw[color=black] (-4.78153796656396,2.7494985824199927) node {$3$};
\draw [fill=black] (-4.598076211353318,1.5) circle (2.5pt);
\draw[color=black] (-4.463678492179733,1.345619237222989) node {$2$};
\draw [fill=black] (-5.464101615137757,2.) circle (2.5pt);
\draw[color=black] (-5.655651521120585,2.1270237784175476) node {$1$};
\draw [fill=black] (0.5980762113533165,2.5) circle (2.5pt);
\draw[color=black] (0.7280262560959774,2.709766148121964) node {$n-2$};
\draw [fill=black] (0.5980762113533152,1.5) circle (2.5pt);
\draw[color=black] (0.4234109264777597,1.2721075267550078) node {$n-1$};
\draw [fill=black] (1.464101615137755,2.) circle (2.5pt);
\draw[color=black] (1.628628100184621,2.0475589098214906) node {$n$};
\end{scriptsize}
\end{tikzpicture}}
\end{center}
\caption{On the left, a straight linear 2-tree with $n$ vertices. On the right, a linear 2-tree with $n$ vertices and single bend at vertex $k$. }
\label{fig:2tree}
\end{figure}


 
 In \cite{Barrett9} network transformations (identical to those found in Section~\ref{sec:s2_basics}) were used to determine the effective resistance in a linear 2-tree with $n$ vertices; the following results were obtained.
\begin{theorem}~\cite[Th. 20]{Barrett9}\label{thm:sl2t}
Let $S_n$ be the straight linear 2-tree on $n$ vertices labeled as in the graph on the left in  Figure~\ref{fig:2tree}. Then for any two vertices $u$ and $v$ of $S_n$ with $u < v$, 
\begin{equation}
r_{S_n}(u,v)=\frac{\sum_{i=1}^{v-u} (F_i F_{i+2u-2}-F_{i-1} F_{i+2u-3})F_{2n-2i-2u+1}}{F_{2n-2}}. \label{eq:resdiststraightsum}
\end{equation}
or equivalently in closed form 

\begin{multline*}r_{S_n}(u,v)	= \frac{F_{m+1}^2+F_{v-u}^2F_{m-2j-v+u+3}^2}{F_{2m+2}}\\
+\frac{F_{m+1}\left[{F_{m-v+u}}((v-u)L_k-F_{v-u})+{F_{m-v+u+1}}\left((v-u-5)F_{v-u+1}+(2v-2u+2)F_{v-u}\right)\right]}{5F_{2m+2}}\end{multline*}
\noindent where $F_p$ is the $p$th Fibonacci number and $L_q$ is the $q$th Lucas number.

\end{theorem}
\noindent Moreover identity~\ref{eq:fibid1} was shown.

In~\cite{Barrett0} the formulas for a straight linear 2-tree were generalized to a linear 2-tree with any number of bends.  See the graph on the right in Figure~\ref{fig:2tree} for an example of a linear 2--tree with a bend at vertex $k$.
The following result is the main result from~\cite{Barrett0} and nicely gives the effective resistance between two vertices in a bent linear 2--tree.
\begin{theorem}~\cite[Th. 3.1]{Barrett0}\label{cor:main2}
 Given a bent linear 2-tree with $n$ vertices, and $p = p_1 + p_2 + p_3$ single bends located at nodes $k_1, k_2,  \ldots, k_p$ and $k_1 < k_2 < \cdots < k_{p-1} < k_p$ the effective resistance between vertices $u$ and $v$ is given by 
	 \begin{multline}\label{eq:genericformres}
	 r_G(u,v)=r_{S_n}(u,v)-\sum_{j=p_1+1}^{p_1+p_2}\Big[F_{k_j-3}F_{k_j}-2\sum_{i=p_1+1}^{j-1}[(-1)^{k_j-k_i+1+j-i}F_{k_i}F_{k_i-3}]+2(-1)^{j+u+k_j}F_{u-1}^2\Big]\cdot\\
	 \Big[F_{n-k_j+2}F_{n-k_j-1}+2(-1)^{v-k_j}F_{n-v}^2\Big]/F_{2n-2}.
	 \end{multline}
 \end{theorem}
\noindent In addition, identity~\ref{eq:wayne1a} was shown.

This paper adds to the growing literature on electrical circuits and recursions by presenting, exploring, and proving results about an infinite array, $C_{i,j}, j \ge 1, 0 \le i \le 2(j-1),$ whose elements are electrical resistances associated with circuits defined on triangular grid graphs.

\section{Some Definitions }\label{sec:s2_basics}

  This section gathers and defines some assorted terms used throughout the paper.
 
\textbf{The (Triangular) $n$-grid.} \cite[Figure 1]{Hendel},\cite[Figure2]{EvansHendel}.
Figure \ref{fig:3grid} is illustrative of the general (triangular) $n$-grid for $n=3.$ As can be seen the $n$-grid consists of $n$ rows with $i, 1 \le i \le n,$ upright oriented triangles arranged in a triangular grid.  Triangles are labeled by row, top to bottom, and diagonal, left to right, as shown in Figure \ref{fig:3grid}.

\begin{figure}[ht!]
\begin{center}
\begin{tabular} {|c|}
\hline
\begin{tikzpicture} [xscale=1,yscale =1]

\node [below] at (3,3.2) {    }; 
 
\draw (0,0)--(1,1)--(2,2)--(3,3);
\draw (2,0)--(3,1)--(4,2);
\draw (4,0)--(5,1);

\draw (3,3)--(4,2)--(5,1)--(6,0);
\draw (2,2)--(3,1)--(4,0);
\draw (1,1)--(2,0);

\draw (0,0)--(2,0)--(4,0)--(6,0);
\draw (1,1)--(3,1)--(5,1);
\draw (2,2)--(4,2);

\begin{scriptsize}
\node  [above] at (1,.2) {$\langle 3,1 \rangle$};  
\node  [above] at (1,.2) {$\langle 3,1 \rangle$};
\node  [above] at (3,.2) {$\langle 3,2 \rangle$};
\node  [above] at (5,.2) {$\langle 3,3 \rangle$};
\node  [above] at (2,1.2) {$\langle 2,1 \rangle$};
\node  [above] at (4,1.2) {$\langle 2,2 \rangle$};
\node  [above] at (3,2.3) {$\langle 1,1 \rangle$};
\end{scriptsize}

\end{tikzpicture}
\\ \hline
\end{tabular}
\caption{A 3-grid with the upright oriented triangles labeled by row and diagonal.}\label{fig:3grid}
\end{center}
\end{figure}
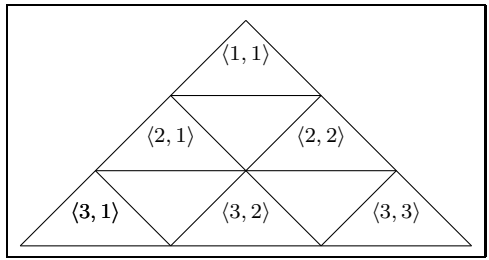

\textbf{The all-one $n$-grid.} Throughout this paper the edge labels of a graph correspond to actual resistance values. The \textit{all-one $n$-grid} refers to an $n$-grid all of whose resistance values are uniformly 1.

We use the notation $T_{r,d,e}$ to refer to the edge label of edge $e, e \in \{L,R,B\}$ (standing, respectively, for the left, right, and base edges of a triangle in the upright oriented position), of the triangle in row $r$ diagonal $d.$ Similarly, $T_{r,d}$ will refer to the triangle in row $r$ diagonal $d.$

Throughout the paper both the all--one $n$-grid and the $m$-grids derived from it ($1 \le m \le n-1$) possess vertical and rotational symmetry (when rotated by $\frac{\pi}{3}).$ \cite[Definition 9.1]{Hendel},\cite[Definition 2.11]{EvansHendel}. 

This symmetry facilitates not presenting results separately for the left, right, and base sides. Typically we will suffice with \textit{the upper left half} of a grid, \cite[Definition 9.6]{Hendel},\cite[Definition 2.12]{EvansHendel}, defined as the set of triangles, $T_{r,d}$ with
$0 \le r \le \lfloor \frac{m+1}{2} \rfloor,$
$1 \le d \le \lfloor\frac{m+2}{2}\rfloor.$

 \begin{example}\label{exa:upperlefthalf}
 If $n=3,$ (see panel A1 in Figure \ref{fig:5panels}) the upper left half consists of the 
 triangles $\langle r,d \rangle, d=1, r=1,2.$
 \end{example} 
 
 The importance of the upper left half is the following result which captures the implications of the symmetry of the $m$-grids \cite[Corollary 9.6]{Hendel},\cite[Lemma 2.14]{EvansHendel}. 
\begin{lemma} \label{lem:upperlefthalf}
For an $m$-grid, once the edge values of the upper half are known, all edge values in the $m$-grid are fixed.
\end{lemma}

\textbf{Corners.} \cite[Equation (29)]{Hendel},\cite[Definition 2.15]{EvansHendel}. Graph--theoretically, a triangle is a corner of an $m$-grid if it has a degree-2 vertex. The 3 corner triangles of an $m$-grid are located at 
$T_{1,1}, T_{m,1}, T_{m,m}.$  For example, for the 3-grid on Figure \ref{fig:3grid}, the three corners are located at $\langle 1,1 \rangle, \langle 3,1 \rangle, \langle 3,3 \rangle.$

\section{The Three Circuit Transformations}\label{sec:circuitfunctions}

As pointed out in Section \ref{sec:s1_circuits},
every circuit has associated with it an underlying labeled graph whose edge labels are electrical resistances. Therefore, to specify an \textit{equivalent circuit transformation} from an initial parent circuit to a transformed child circuit we must specify the vertex, edge, and label transformations. By equivalent circuit transformation we mean one that maintains the effective resistance between vertices that appear in both parent and child circuit. There are three basic circuit transformations that we use that preserve effective resistance: \textit{series, $\Delta-
Y$, and $Y-\Delta.$} Figure \ref{fig:seriesparallel} illustrates the series transformation . 

 The following are the key points about this transformation.
\begin{itemize}
    \item The top parent graph has 3 nodes and 2 edges
    \item The transformed child graph below has 2 nodes and one edge
    \item There is a formula,\cite[pg. 43]{Koshy} $R_1+R_2$ giving the edge label of the child graph in terms of the edge labels of the parent graph.
\end{itemize}

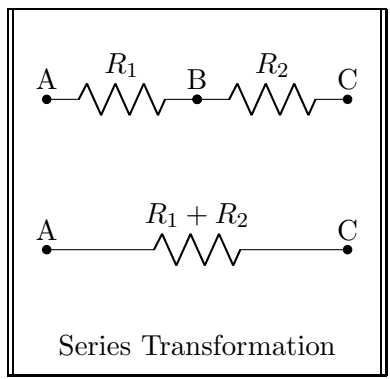
\begin{figure}[ht!]
\begin{center}
\begin{tabular}{||c||}
\hline
\begin{circuitikz}

\node [below] at (3,3.2) {    };
\node at (0,3) { };

\draw (0,2) to [R, *-*,l=$R_1$] (2,2); 
\draw (2,2) to [R, *-*,l=$R_2$] (4,2); 
\coordinate (A)  at (0,2); 
\coordinate (B)   at (2,2);
\coordinate (C)   at (4,2);

\node at (A) [above] {A};
\node at (B) [above] {B};
\node at (C) [above] {C};

\draw (0,0) to [R, *-*,l=$R_1+R_2$] (4,0); 
\coordinate  (AA) at (0,0);  
\coordinate (CC) at (4,0);

\node at (AA) [above] {A};
\node at (CC)[above] {C};

\node [below] at (2,-1) {Series Transformation};;
\end{circuitikz} 
\\ \hline

\end{tabular}

\end{center}
\caption{Illustration of the series transformations. See narrative for further details. }\label{fig:seriesparallel}
\end{figure}

The remaining two circuit transformations are the $\Delta-Y$ transformation which transforms a parent simple 3-edge loop to a claw (3-edge outstar), and the $Y-\Delta$ transformation which takes a claw to a 3-edge loop,\cite[Figure 2]{Hendel}, \cite[Definition 2.4]{EvansHendel}. The relevant transformation functions are
\begin{equation}\label{equ:deltay} \Delta(x,y,z) = \frac{xy}{x+y+z}; \qquad 
Y(x,y,z) = \frac{xy+yz+zx}{x}.\end{equation}

Following the computations presented in this paper will not require details of these transformations or how the order of arguments relates to the underlying graphs.  To follow the computations needed in this paper it suffices to know the four circuit transformation functions presented in Section \ref{sec:proofmethods}.

\section{The Reduction Algorithm}\label{sec:reduction}

This section presents the basic reduction algorithm. This algorithm was first presented in \cite[pg. 18]{Barrett0}  where the algorithm was used for purposes of proof but not used computationally, since computations were done using the combinatorial Laplacian. Hendel \cite[Definition 2.3,Figure 3]{Hendel}  was the first to use the reduction algorithm computationally.  Moreover, \cite[Algorithm 2.8, Figure 3 and Section 4]{EvansHendel}  was the first to show that four transformation functions suffice for all computations. These four circuit transformation functions will be presented in Section \ref{sec:proofmethods}; knowledge of them suffices to follow, and be able to reproduce, all computations presented in this paper.  The usefulness of this algorithm in uncovering patterns is alluded to in \cite{Hendel, EvansHendel}.

We begin the presentation of the four circuit transformations with some basic illustrations.

\color{black}
The reduction algorithm takes a parent $m$ grid and \textit{reduces} it, by removing one row of triangles, to a child $m-1$ grid.
Figure \ref{fig:5panels}, illustrates the five steps in reducing the 3 grid (Panel A) to a two grid (Panel E), \cite[Steps A-E, Figure 3]{Hendel}, \cite[Algorithm 2.8]{EvansHendel}  
\begin{itemize}
    \item Step 1 - Panel A: Start with a labeled 3-grid
    \item Step 2 - Panel B: Apply a $\Delta-Y$ transformation to each upright triangle (a 3-loop) resulting in a grid of 3 rows of 3-stars, as shown.
    \item Step 3 - Panel C: Discard the corner tails, edges with a vertex of degree one. This does not affect the resistance labels of edges in the reduced two grid in panel E. (However, these corner tails are useful for computing effective resistance as shown in
    \cite{Barrett0,Evans2022}).
    \item Step 4 - Panel D: Perform series transformations on all consecutive pairs of boundary edges (i.e., the dashed edges in panel C).
    \item Step 5 - Panel E: Apply $Y-\Delta$ transformations to all remaining claws, transforming them into 3-loops.
\end{itemize}

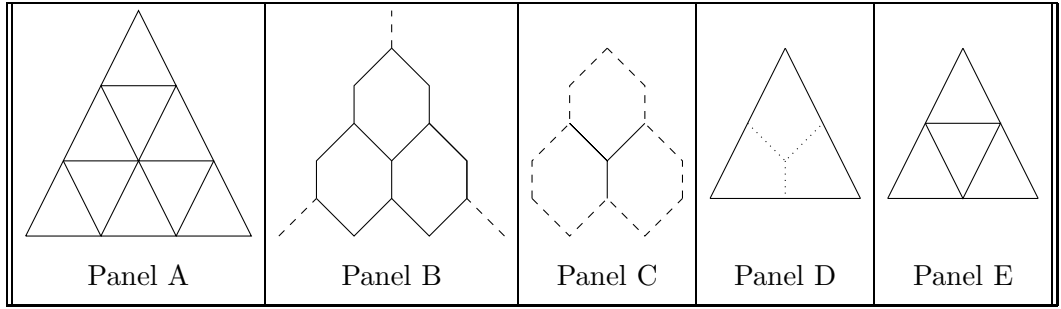
\begin{figure}[ht!]
\begin{center}
\begin{tabular}{||c|c|c|c|c||}
\hline
\begin{tikzpicture}[yscale=.5,xscale=.5]

\node [below] at (3,3.2) {    }; 
 
\draw (0,0)--(1,2)--(2,0)--(0,0);
\draw (2,0)--(3,2)--(4,0)--(2,0);
\draw (4,0)--(5,2)--(6,0)--(4,0); 
\draw (1,2)--(2,4)--(3,2)--(1,2);
\draw (3,2)--(4,4)--(5,2)--(3,2);
\draw (2,4)--(3,6)--(4,4)--(2,4);

\node [below] at (3,6.2) {     };
\node [below] at (3,-.5) {Panel A};
\end{tikzpicture} 
& 
\begin{tikzpicture}[xscale=.5,yscale=.5] 
\draw  (1,1)--(1,2)--(2,3)--(2,4)--(3,5);
\draw  (5,1)--(5,2)--(4,3)--(4,4)--(3,5);
\draw (1,1)--(2,0)--(3,1)--(3,2)--(2,3);
\draw (3,1)--(4,0)--(5,1)--(5,2)--(4,3)--(3,2); 
 
\draw [dashed] (6,0)--(5,1);
\draw [dashed] (0,0)--(1,1);
 \draw [dashed] (3,5)--(3,6);
 
\node [below] at (3,-.5) {Panel B};
\end{tikzpicture}&
\begin{tikzpicture}[yscale=.5,xscale=.5]
 
\draw [dashed] (1,1)--(1,2)--(2,3)--
(2,4)--(3,5);
\draw [dashed] (5,1)--(5,2)--(4,3)--(4,4)--(3,5);
\draw [dashed] (1,1)--(2,0)--(3,1)--(4,0)--(5,1);
\draw (3,1)--(3,2)--(2,3)--(3,2)--(4,3);

\node [below] at (3,-.5) {Panel C};
 
\end{tikzpicture}&
\begin{tikzpicture}[yscale=.5,xscale=.5]
 
\draw (1,1)--(2,3)--(3,5);
\draw (5,1)--(4,3)--(3,5);
\draw (1,1)--(3,1)--(5,1);
\draw [dotted] (2,3)--(3,2)--(4,3);
\draw [dotted] (3,2)--(3,1); 
 
 \node [below] at (3,-.5) {Panel D};
\end{tikzpicture}&
\begin{tikzpicture}[yscale=.5,xscale=.5]
 
\draw (1,1)--(2,3)--(3,5);
\draw (5,1)--(4,3)--(3,5);
\draw (1,1)--(3,1)--(5,1);
\draw (2,3)--(4,3)--(3,1)--(2,3); 
 
 \node [below] at (3,-.5) {Panel E};
 
\end{tikzpicture}
\\ \hline
\end{tabular}

\end{center}
\caption{Illustration of the reduction algorithm, on a 3-grid. The panel labels correspond to the five steps indicated in the narrative.}\label{fig:5panels}
\end{figure}
\color{black}
 
The important point here is that each of the five steps involves specific circuit transformations. However, to follow, and be able to reproduce the computations in this paper, only the four circuit transformation functions presented in the next section are needed. The derivation of these four circuit transformation functions is not needed and has been given in detail in the references cited. An example at the end of this section illustrates what is needed.  

In the sequel, we will typically start with an all--one $n$-grid and successively apply the reduction algorithm resulting in a collection of $m$ grids, $1 \le m \le n-1.$ The notation
\begin{multline*} T_{r,d,X}^m, X \in \{L,R,B, LR\} \text{ indicates the resistance label of side $X$}\\ \text{in triangle $T_{r,d}$ of the all--one $n$-grid reduced $m$ times}\\ \text{The symbol LR will be used in a context} \text{ when the side depends on the parity of a parameter. }
\end{multline*}

Additionally, if we deal with a single reduction we may use the superscripts $p,c$ to distinguish between the parent grid and the child grid when the actual number of reductions used is not important.
 
\begin{example}\label{exa:suffices} Referring to Figure \ref{fig:5panels}, the function \emph{left} presented in the next section takes the 9 resistance edge-labels of triangles  $T_{2,1}^p, T_{2,2}^p,T_{3,2}^p$ in the parent 3-grid in Panel A and computes the resistance edge-value, $T_{2,2,L}^c$ of the child 2-grid in Panel E. Thus the four transformation functions of the next section suffice to verify and reproduce the computations in this paper.
\end{example}

\section{The Four Transformation Functions.}\label{sec:proofmethods}

 As mentioned in Example \ref{exa:suffices} and the surrounding narrative, this section presents the four circuit transformation functions that suffice to follow and reproduce the computations presented in this paper  \cite[Section 4]{EvansHendel}:
\begin{itemize}
    \item Boundary edges 
    \item Base (non-boundary) edges 
    \item Right (non-boundary)edges 
    \item Left (non-boundary) edges 
\end{itemize}

We begin our description of the four transformation functions with the base edge case. Illustrations are based on Figure \ref{fig:3grid}.  We first illustrate with the base edge of the top corner triangle in Figure \ref{fig:3grid} and then generalize. Note, that the $\Delta$ and $Y$ functions have been defined in \eqref{equ:deltay}. 

We have
$$
T_{1,1,B}^c =
    Y(\Delta(T_{3,2,L}^{p},
            T_{3,2,R}^{p},
            T_{3,2,B}^{p}),
            \Delta(T_{2,1,R}^{p},
            T_{2,1,B}^{p},
            T_{2,1,L}^{p}),
             \Delta(T_{2,2,B}^{p},
            T_{2,2,L}^{p},
            T_{2,2,R}^{p})).
$$
This is a function of 9 variables. At times it becomes convenient to emphasize the triangles involved. We will use the following notation to indicate the dependency on triangles.
$$
T_{1,1,B}^c =F(T_{3,2}^{p},
                T_{2,1}^{p},
                T_{2,2}^{p}),
$$
which is interpreted as saying \textit{the base edge of $T_{1,1}^c$ is some function ($F$) of the edge-labels of the triangles $T_{3,2}^p,T_{2,1}^p,T_{2,2}^p.$} 
Clearly, this notation is mnemonical and cannot be used computationally. It is however very useful in proofs as will be seen later.

The previous two equations can be generalized to an arbitrary $m$-grid and arbitrary row and diagonal (with minor constraints, $r+2 \le n, d+1 \le r,$ on the row and diagonal).  We have 
\begin{multline*}
T_{r,d,B}^c =
    Y(\Delta(T_{r+2,d+1,L}^{p},
            T_{r+2,d+1,R}^{p},
            T_{r+2,d+1,B}^{p}),
            \Delta(T_{r+1,d,R}^{p},
            T_{r+1,d,B}^{p},
            T_{r+1,d,L}^{p}),\\
             \Delta(T_{r+1,dr+1,B}^{p},
        T_{r+1,d+1,L}^{p},
            T_{r+1,d+1,R}^{p})).
\end{multline*}
and
$$
T_{r,d,B}^c =F(T_{r+2,d+1}^{p},
                T_{r+1,d}^{p},
                T_{r+1,d+1}^{p}).
$$

We next list the remaining three transformation functions. 

For $r+2 \le n, d+1 \le r,$ for a boundary left edge we have 
 
\begin{equation*}
T_{r,1,L}^c =
    \Delta(T_{r,1,B}^{p},
            T_{r,1,L}^{p},
            T_{r,1,R}^{p})+
            \Delta(T_{r+1,1,L}^{p},
            T_{r+1,1,R}^{p},
            T_{r+1,1,B}^{p}) ,
\end{equation*}
and  
$$
T_{r,1,L}^c =F(T_{r,1}^{p},
                T_{r+1,1}^{p}).
$$

For $r+2 \le n, d+1 \le r,$ for non boundary left edges we have,
    \begin{multline}\label{equ:leftside9proofs} 
T_{r,d,L}^c =
    Y(\Delta(T_{r,d-1,R}^{p},
            T_{r,d-1,B}^{p},
            T_{r,d-1,L}^{p}),
     \Delta(T_{r,d,B}^{p},
        	T_{r,d,L}^{p},
            T_{r,d,R}^{p}),\\  
            \Delta(T_{r+1,d,L}^{p},
            T_{r+1,d,R}^{p},
            T_{r+1,d,B}^{p}))
        \end{multline}
and  
\begin{equation}\label{equ:leftside3proofs}
T_{r,d,L}^c =F(T_{r,d-1}^{p},
                T_{r,d}^{p},
                T_{r+1,d}^{p}).
\end{equation}

For  $r+1  \le n,  2 \le d  \le r-1,$ for the right sides,
\begin{multline*}
T_{r,d,R}^c =
    Y(\Delta(T_{r,d,B}^{p},
            T_{r,d,L}^{p},
            T_{r,d,R}^{p}),
      \Delta(T_{r+1,d,L}^{p},
        T_{r+1,d,R}^{p},
            T_{r+1,d,B}^{p},\\      
      \Delta(T_{r,d-1,R}^{p},
            T_{r,d-1,B}^{p},
            T_{r,d-1,L}^{p}))
\end{multline*}
and 
$$
T_{r,d,R}^c =F(T_{r,d}^{p},
                T_{r+1,d}^{p},
                T_{r,d-1}^{p}).
$$

\begin{comment}
Notice that we only defined the boundary function for the left boundary ($d=1$). Similarly, notice that for example the base edge function requires $r \le n-2.$ This is not a restriction. For by Lemma \ref{lem:upperlefthalf}, once the upper left half is calculated, the remaining edge values follow by symmetry considerations. Thus the above functions with their restrictions do indeed suffice. 
\end{comment}


\section{Computational Examples}\label{sec:Appendix_A}

The four transformation functions of Section \ref{sec:proofmethods} with up to 9 arguments may appear computationally challenging. The purpose of this section is to illustrate their computational use.  Additionally, the results computed will be used both to motivate and prove the main theorem.

\subsection{One Reduction of an all--one $n$--grid}
 An all--one $n$ grid definitionaly has uniform labels of 1. Hence, we may calculate the edge resistance values in an $n-1$ grid arising from one reduction of the all--one $n$ grid as follows:
\begin{itemize}
    \item $T_{r,1,L}=\Delta(r,1,1)+\Delta(r+1,1,1)=\frac{2}{3}, 1 \le r \le n-1.$
    \item The preceding bullet computes resistance labels for the left boundary. By Lemma \ref{lem:upperlefthalf}, and by symmetry considerations the same computed value holds on the other two grid boundary edges: $T_{r,r,R} = T_{n-1,r,B} = \frac{2}{3}, 1 \le r \le n-1.$
    \item All other edge values are 1, since
    the computation $Y(\Delta(1,1,1), \Delta(1,1,1), \Delta(1,1,1))=1$ applies to $T_{r,d,X}, X \in \{L,R,B\}$
    \item Again, cases not covered by the four transformation functions are covered by symmetry considerations and Lemma \ref{lem:upperlefthalf}. For example the formula for $T_{r,d,B}$ is only valid for $r \le n-2,$ and therefore both the symmetry considerations and the lemma are needed.  
\end{itemize}

We may summarize our results in a lemma, see also, \cite[Corollary 5.1]{Hendel},\cite[Lemma 6.1]{EvansHendel}.

\begin{lemma}\label{lem:1reduction}
The resistance labels of the $n-1$ grid arising from one reduction of an all--one $n$ grid are as follows:
\begin{enumerate}
    \item Boundary resistance labels are uniformly $\frac{2}{3}$.
    \item Interior resistance labels are uniformly 1.
\end{enumerate}
\end{lemma}

The top corner triangle, $T_{1,1}$ of the $n-1$
grid is presented in Panel A of Figure \ref{fig:motivationillustration}.

\subsection{Uniform Central Regions}
Prior to continuing with the computations we  introduce the concept of the uniform center which will be used in the proof of the main theorem.

First, we can identify a triangle with the ordered list, Left, Right, Base, of its resistance labels. Two triangles are then equivalent if their edge labels are equal. By Lemma \ref{lem:1reduction} for the once reduced $n-1$ grid we have
$$
    T_{r,1}=\left(\frac{2}{3},1,1\right) \text{  and  }
    T_{r,r} = \left(1,\frac{2}{3},1\right) \text{  for  }
    2 \le r \le n-1.
$$
Although $T_{r,1}$ and $T_{r,r}$ are not strictly equivalent we will say they are equivalent up to symmetry since each triangle may be derived from the other by a vertical symmetry, \cite[Definition 5.8]{EvansHendel}.

Using these concepts of triangle equivalence and triangle equality up to symmetry, we note that the central region, $2 \le r \le n-1$ of diagonal 1 of the reduced $n-1$ grid is uniform, that is all triangles are equal. We also note that the interior of the reduced $n-1$ grid is uniform.   

This presence of uniformity generalizes. The formal statement of the uniform center \cite[Theorem 6.2]{EvansHendel} is as follows:

 \begin{theorem}[Uniform Center]\label{the:uniformcenter} For any  $s \ge 1,$ let  $n \ge 4s,$ and $1 \le d \le s:$
 \begin{enumerate}
\item For
 \begin{equation}\label{equ:uniformcenter}
 s+ d \le r \le m-2s
 \end{equation}
 the triangles $T_{r,d}^s$ are all equal.
\item For 
\begin{equation}\label{equ:uniformcenter2}
2s-1 \le r \le m-2s
\end{equation}
the left sides $T_{r,s,L}^s$ are all equal,
$T_{2s-1,s,R}^s=T_{2s-1,s,L}^s,$
$T_{r,s,R}^s=1, 2s \le r \le m-2s,$
and $T_{r,s,B}^s=1, 2s-1 \le r \le m-2s-1.$
\item For any triangle in the uniform center, that is, satisfying \eqref{equ:uniformcenter},
$ T_{r,d,R}=T_{r,d,B}.$
\end{enumerate}
\end{theorem}

This theorem has an elegant graphical interpretation. It states that the sub triangular grid whose corner triangles are $T_{2s-1s}^s, T_{m-2s,m}^s, T_{m-2s, m-2s}^s$ has interior labels of 1 and a single uniform label along its edge boundary. However this interpretation is not needed in the sequel.

\subsection{Two Reductions of an all--one $n$--grid} 

We continue illustrating computations by considering the $n-2$ grid arising from 2 reductions of the all--one $n$ grid (or one reduction of the $n-1$ grid.)  

By \eqref{equ:leftside3proofs}
$$
    T_{3,2,L}^2 = F(T_{3,1}^1, T_{3,2}^1, T_{4,2}^1).
$$
By Lemma \ref{lem:1reduction}, we have
$$
    T_{3,1}^1=(\frac{2}{3},1,1),
    T^1_{3,2}=T^p_{4,2} =(1,1,1).
$$
Hence, by \eqref{equ:leftside9proofs}
\begin{equation}\label{equ:t322627}
T_{3,2,L}^c=Y(\Delta(1,1,\frac{2}{3}),
            \Delta(1,1,1), \Delta(1,1,1))=
            \frac{26}{27},
\end{equation}
as shown in Panel B of Figure \ref{fig:motivationillustration}.

To continue with the computations  
we define the function,
\begin{equation}\label{equ:g0}
    G_0(X)=Y(\Delta(1,1,X),
            \Delta(1,1,1), \Delta(1,1,1))=
            \frac{X+8}{9},
\end{equation}
and confirm $G_0(\frac{2}{3}) = \frac{26}{27}.$

\subsection{Three Reductions of an all--one $n$-grid.}
We next compute $T_{5,3,L}^3.$ Continuing as in the case of $T_{3,2,1},$ we have by
\eqref{equ:leftside3proofs}
\begin{equation}\label{equ:temp1}
T_{5,3,L}^3=F(T_{5,2}^2, T_{5,3}^2, T_{6,3}^2).
\end{equation}
By Theorem \ref{the:uniformcenter}(b)
$$
    T_{5,2,L}^2 = T_{3,2,L}^2,
$$
and by \eqref{equ:t322627}
$$
    T_{3,2,L}^2 = \frac{26}{27},
$$
implying 
$$
    T_{5,2,L}^2 = \frac{26}{27}.
$$
Again, by Theorem \ref{the:uniformcenter}
all resistance labels of $T_{5,3}^2, T_{6,3}^2$ are 1. Plugging this into \eqref{equ:temp1} and using \eqref{equ:leftside9proofs} and \eqref{equ:g0},
we have
$$
T_{5,3,L}^3 = Y\left(\Delta\left(\frac{26}{27},1,1\right), \Delta(1,1,1), \Delta(1,1,1)\right)=G_0\left(\frac{26}{27}\right)=\frac{242}{243}
$$
Panel C of Figure \ref{fig:motivationillustration} illustrates this.

We can continue this process inductively. For example, $T^4_{7,4} = G_0(\frac{242}{243}).$ The result is summarized as follows.

\begin{lemma} \label{lem:row0}
With $G_0(X)$ defined by \eqref{equ:g0} we have
$T_{1,1,L}^1=\frac{2}{3}$ and for $s \ge 2,$ 
$T_{2s-1,s,L}=G_{0}(T_{2s-3,s-1,L}).$ 
\end{lemma}

An almost identical argument using the circuit transformations for the right edge   shows the following.

 \begin{lemma} \label{lem:row1}
 Let $G_1(X)=\frac{1}{3} \frac{X+8}{X+2}.$ Then for $k \ge 0,$
$T_{3+2k,2+k,R}^{2+k}=G_1(T^k_{1+2k,1+k})$
\end{lemma}

\section{Motivation for the Circuit Array}\label{sec:motivation}

This section motivates the underlying construction of the Circuit Array. We initialize with an all--one $n$-grid for $n$ large enough. As computed in Section \ref{sec:Appendix_A}, we have:
\begin{itemize}
    \item 
    $T_{1,1,L}^1=\frac{2}{3}=1-\frac{3}{9^1}.$ 
    See Panel A of Figure \ref{fig:motivationillustration}.
    \item $T_{3,2,L}^2=\frac{26}{27}=1-\frac{3}{9^2}.$ See Panel B of Figure \ref{fig:motivationillustration}.
    \item $T_{5,3,L}^1=\frac{242}{243}=1-\frac{3}{9^3}.$ See Panel C of Figure \ref{fig:motivationillustration}.
\end{itemize}

The resulting sequence
$$
	\frac{2}{3}, \frac{26}{27}, \frac{242}{243}, \dotsc
$$
satisfies 
\begin{equation}\label{equ:row0}
	1 -\frac{3}{9^s}, s=1,2,3, \dotsc.
\end{equation}
In this particular case the denominators, $9^s$ form a linear homogeneous recursion with constant coefficients (LHRCC) of order 1,
$$
        G_s = 9 G_{s-1}, s \ge 1, \qquad G_0=1.
$$
Similarly, the numerators satisfy the LRCC,
$$
    G_s =3G_{s-1}+8, s \ge 1, \qquad G_0=-2. 
$$
(and therefore, since a sequence satisfying a linear, non--homogeneous recursion with constant coefficients (LRCC) will also satisfy an 
LHRCC albeit with a higher degree), the sequence also satisfies an LHRCC.

The sequence just studied forms row 0 of the Circuit Array, Table \ref{tab:circuitarray}.  To determine row 1 of the Circuit Array we compute the following:
\begin{itemize}
    \item The right side of the triangle left-adjacent to the top corner triangle of the 2-rim of reduction 2 has label
    $\frac{13}{12} = 1+\frac{2}{3} \frac{1}{9^{2-1}-1},$ as shown in Panel B of Figure \ref{fig:motivationillustration}.
    \item The right side of the triangle left-adjacent to the top corner triangle of the 3-rim of reduction 3 has label
    $\frac{121}{120} = 1+\frac{2}{3} \frac{1}{9^{3-1}-1},$ as shown in Panel C of Figure \ref{fig:motivationillustration}.
    \item The right side of the triangle left-adjacent to the top corner triangle of the 4-rim of reduction 4 has label
    $\frac{1093}{1092} = 1+\frac{2}{3} \frac{1}{9^{4-1}-1}.$
\end{itemize}

  The resulting sequence
$$
	\frac{13}{12}, \frac{121}{120},   \frac{1093}{1092},	\frac{9841}{9840}, \dotsc
$$
satisfies $1+\frac{2}{3 } { 9^{s-1}-1     }.$ We again see the presence of LRCC. The sequence of twice the denominators satisfies the LRCC
$$
    G_{s+1} = 9G_s +24, s \ge 3, G_2 =24,
$$
while the sequence of twice the numerators satisfies the LRCC,
$$
    G_{s+1} = 9G_s +8, s \ge 3, G_2 =26,
$$
and hence both numerators and denominators satisfy LHRCC, albeit of higher order.

These calculations determine the construction of the Circuit Array by rows. Figure~\ref{fig:motivationillustration} can be used to motivate a construction by columns. Each perspective provides different sequences.
\begin{itemize}
    \item As shown in Panel A, Column 1, consists of the singleton $\frac{2}{3}.$ We may describe the process of generating this column by starting, at the left resistance edge label of triangle $T_{1,1},$ where 1 corresponds to the number of underlying reductions of the all--one $n$-grid,  traversing to the left (in this case there is nothing further to transverse) and ending at the left--most edge of the underlying row. This singleton $\frac{2}{3}$ is column 0 of the Circuit Array, Table \ref{tab:circuitarray}.
     \item As shown in Panel B, Column 2 may be obtained as follows:  Start, at the left resistance-label of   triangle $T_{3,2},$ where the number of reductions of of the all--one $n$-grid for this column is $2,  \text{ and }  3 = 2 \times 2 -1$. This resistance is  $\frac{26}{27}.$ Then traverse to the left,   and end at the left--most edge of the underlying row. By recording the labels during this transversal we obtain $\frac{26}{27},\frac{13}{12}, \frac{1}{2},$ which is column 1 of the Circuit Array, Table \ref{tab:circuitarray}, starting at row 0 and ending at row 2.
     \item As shown in Panel C, Column 3 may be obtained as follows:  Start,  at the left resistance-label of   triangle $T_{5,3},$ where the number of reductions of the all--one $n$-grid for this column is $3, \text{ and } 5 = 2 \times 3 -1.$ This resistance  is
      $\frac{242}{243},$ traverse to the left,   and end at the left most edge of the underlying row. By recording the labels during this transversal we obtain  $\frac{242}{243},\frac{121}{120}, \frac{89}{100},\frac{1157}{960},\frac{13}{32},$ which is column 2 of the Circuit Array, Table \ref{tab:circuitarray}), starting at row 0 and ending at row 4.
     \item The above suggests in general, that column $c \ge 1$ of the Circuit Array will consist of the resistance labels of the left and right sides of the triangles $T_{2c-1,i}^c, i=c, c-1, \dotsc 1.$ This will be formalized in the next section.
\end{itemize}

\begin{figure}[ht!]
\begin{center}
\begin{tabular}{||c|c|c||}
\hline 

\begin{tikzpicture}[scale = 1, line cap=round,line join=round,>=triangle 45,x=1.0cm,y=1.0cm,scale = 1]
\draw [line width=.8pt,color=black] (-2.,10.)-- (-3.,11.732050807568879);
\draw [line width=.8pt,color=black] (-2.,10.)-- (-4.,10);
\draw [line width=.8pt,color=black] (-3.,11.732050807568879)-- (-4.,10.);

\begin{scriptsize}
\draw[color=black] (-3.8,11) node {$\dfrac{2}{3}$};
\draw[color=black] (-2.2,11) node {$\dfrac{2}{3}$};
)
 \node [above] at (-3,9.5) {Panel A}; 
 
\end{scriptsize}
\end{tikzpicture} &
\begin{tikzpicture}[scale = 1, line cap=round,line join=round,>=triangle 45,x=1.0cm,y=1.0cm,scale = 1]
\draw [line width=.8pt,color=black] (-2.,10.)-- (-3.,11.732050807568879);
\draw [line width=.8pt,color=black] (-2.,10.)-- (-4.,10);
\draw [line width=.8pt,color=black] (-3.,11.732050807568879)-- (-4.,10.);
\draw [line width=.8pt,color=black] (-4.,10.)-- (-3.,8.267949192431121);
\draw [line width=.8pt,color=black] (-3.,8.267949192431121)-- (-2.,10.);
\draw [line width=.8pt,color=black] (-4.,10.)-- (-5.,8.267949192431123);
\draw [line width=.8pt,color=black] (-5.,8.267949192431123)-- (-3.,8.267949192431121);
\draw [line width=.8pt,color=black] (-3.,8.267949192431121)-- (-1.,8.267949192431121);
\draw [line width=.8pt,color=black] (-1.,8.267949192431121)-- (-2.,10.);
\draw [line width=.8pt,color=black] (-3.,8.267949192431121)-- (-2.,6.535898384862243);
\draw [line width=.8pt,color=black] (-3.,8.267949192431121)-- (-4.,6.535898384862244);
\draw [line width=.8pt,color=black] (-1.,8.267949192431121)-- (-2.,6.535898384862244);
\draw [line width=.8pt,color=black] (-4.,6.535898384862244)-- (-2.,6.535898384862243);
\draw [line width=.8pt,color=black] (-3.,8.267949192431121)-- (-5.,8.267949192431121);
\draw [line width=.8pt,color=black] (-5.,8.267949192431121)-- (-4.,6.535898384862244);
\draw [line width=.8pt,color=black] (-2.,6.535898384862243)-- (0.,6.535898384862242);
\draw [line width=.8pt,color=black] (-4.,6.535898384862243)-- (-6.,6.535898384862242);
\draw [line width=.8pt,color=black] (0.,6.535898384862242)-- (-1.,8.267949192431121);
\draw [line width=.8pt,color=black] (-5.,8.267949192431123)-- (-6.,6.535898384862246);

\begin{scriptsize}
\draw[color=black] (-5.8,7.5) node {$\dfrac{1}{2}$};
\draw[color=black] (-4.9,7.5) node {$\dfrac{13}{12}$};
\draw[color=black] (-3.8,7.5) node {$\dfrac{26}{27}$};
\draw[color=black] (-2.2,7.5) node {$\dfrac{26}{27}$};

 \node [above] at (-3,6) {Panel B}; 

\end{scriptsize}
\end{tikzpicture} &
\scalebox{.65}{\begin{tikzpicture}[scale = 1, line cap=round,line join=round,>=triangle 45,x=1.0cm,y=1.0cm,scale = 1]
\draw [line width=.8pt,color=black] (-2.,10.)-- (-3.,11.732050807568879);
\draw [line width=.8pt,color=black] (-2.,10.)-- (-4.,10);
\draw [line width=.8pt,color=black] (-3.,11.732050807568879)-- (-4.,10.);
\draw [line width=.8pt,color=black] (-4.,10.)-- (-3.,8.267949192431121);
\draw [line width=.8pt,color=black] (-3.,8.267949192431121)-- (-2.,10.);
\draw [line width=.8pt,color=black] (-4.,10.)-- (-5.,8.267949192431123);
\draw [line width=.8pt,color=black] (-5.,8.267949192431123)-- (-3.,8.267949192431121);
\draw [line width=.8pt,color=black] (-3.,8.267949192431121)-- (-1.,8.267949192431121);
\draw [line width=.8pt,color=black] (-1.,8.267949192431121)-- (-2.,10.);
\draw [line width=.8pt,color=black] (-3.,8.267949192431121)-- (-2.,6.535898384862243);
\draw [line width=.8pt,color=black] (-3.,8.267949192431121)-- (-4.,6.535898384862244);
\draw [line width=.8pt,color=black] (-1.,8.267949192431121)-- (-2.,6.535898384862244);
\draw [line width=.8pt,color=black] (-4.,6.535898384862244)-- (-2.,6.535898384862243);
\draw [line width=.8pt,color=black] (-3.,8.267949192431121)-- (-5.,8.267949192431121);
\draw [line width=.8pt,color=black] (-5.,8.267949192431121)-- (-4.,6.535898384862244);
\draw [line width=.8pt,color=black] (-2.,6.535898384862243)-- (0.,6.535898384862242);
\draw [line width=.8pt,color=black] (0.,6.535898384862242)-- (-1.,8.267949192431121);
\draw [line width=.8pt,color=black] (-5.,8.267949192431123)-- (-6.,6.535898384862246);
\draw [line width=.8pt,color=black] (-6.,6.535898384862246)-- (-4.,6.535898384862244);
\draw [line width=.8pt,color=black] (-2.,6.535898384862243)-- (-1.,4.803847577293364);
\draw [line width=.8pt,color=black] (-1.,4.803847577293364)-- (0.,6.535898384862242);
\draw [line width=.8pt,color=black] (-4.,6.535898384862244)-- (-3.,4.8038475772933635);
\draw [line width=.8pt,color=black] (-3.,4.8038475772933635)-- (-2.,6.535898384862243);
\draw [line width=.8pt,color=black] (-6.,6.535898384862246)-- (-5.,4.803847577293364);
\draw [line width=.8pt,color=black] (-5.,4.803847577293364)-- (-4.,6.535898384862244);
\draw [line width=.8pt,color=black] (-6.,6.535898384862246)-- (-7.,4.80384757729337);
\draw [line width=.8pt,color=black] (1.,4.803847577293362)-- (0.,6.535898384862242);
\draw [line width=.8pt,color=black] (-7.,4.80384757729337)-- (-5.,4.803847577293364);
\draw [line width=.8pt,color=black] (-3.,4.80384757729337)-- (-5.,4.803847577293364);
\draw [line width=.8pt,color=black] (-3.,4.80384757729337)-- (-1.,4.803847577293364);
\draw [line width=.8pt,color=black] (-1.,4.803847577293364)-- (1.,4.803847577293362);
\draw [line width=.8pt,color=black] (-1.,4.803847577293364)-- (0.,3.071796769724484);
\draw [line width=.8pt,color=black] (0.,3.071796769724484)-- (1.,4.803847577293362);
\draw [line width=.8pt,color=black] (-3.,4.8038475772933635)-- (-2.,3.071796769724486);
\draw [line width=.8pt,color=black] (-2.,3.071796769724486)-- (-1.,4.803847577293364);
\draw [line width=.8pt,color=black] (-5.,4.803847577293364)-- (-4.,3.0717967697244837);
\draw [line width=.8pt,color=black] (-4.,3.0717967697244837)-- (-3.,4.8038475772933635);
\draw [line width=.8pt,color=black] (-5.,4.803847577293364)-- (-7.,4.80384757729337);
\draw [line width=.8pt,color=black] (-7.,4.80384757729337)-- (-6.,3.071796769724488);
\draw [line width=.8pt,color=black] (-6.,3.071796769724488)-- (-5.,4.803847577293364);
\draw [line width=.8pt,color=black] (-7.,4.80384757729337)-- (-8.,3.071796769724494);
\draw [line width=.8pt,color=black] (-8.,3.071796769724494)-- (-6.,3.071796769724488);
\draw [line width=.8pt,color=black] (0.,3.071796769724484)-- (2.,3.0717967697244815);
\draw [line width=.8pt,color=black] (2.,3.0717967697244815)-- (1.,4.803847577293362);
\draw [line width=.8pt,color=black] (-1.,4.803847577293364)-- (-2.,3.0717967697244877);
\draw [line width=.8pt,color=black] (-2.,3.0717967697244877)-- (0.,3.071796769724484);
\draw [line width=.8pt,color=black] (-3.,4.8038475772933635)-- (-4.,3.0717967697244863);
\draw [line width=.8pt,color=black] (-4.,3.0717967697244863)-- (-2.,3.071796769724486);
\draw [line width=.8pt,color=black] (-5.,4.803847577293364)-- (-6.,3.0717967697244846);
\draw [line width=.8pt,color=black] (-6.,3.0717967697244846)-- (-4.,3.0717967697244837);

\begin{scriptsize}
\draw[color=black] (-7.8,4.2) node {$\dfrac{13}{32}$};
\draw[color=black] (-6.8,3.6) node {$\dfrac{1157}{960}$};
\draw[color=black] (-5.8,4.2) node {$\dfrac{89}{100}$};
\draw[color=black] (-4.7,3.6) node {$\dfrac{121}{120}$};
\draw[color=black] (-3.8,4.2) node {$\dfrac{242}{243}$};
\draw[color=black] (-2.2,4.2) node {$\dfrac{242}{243}$};

\end{scriptsize}
 \node [above] at (-3,2.5) {Panel C}; 
\end{tikzpicture}
}
\\ \hline
\end{tabular}

\end{center}
\caption{Graphical illustration showing locations of various edge resistance values computated in this section. See the narrative for further details. }\label{fig:motivationillustration} 
\end{figure}
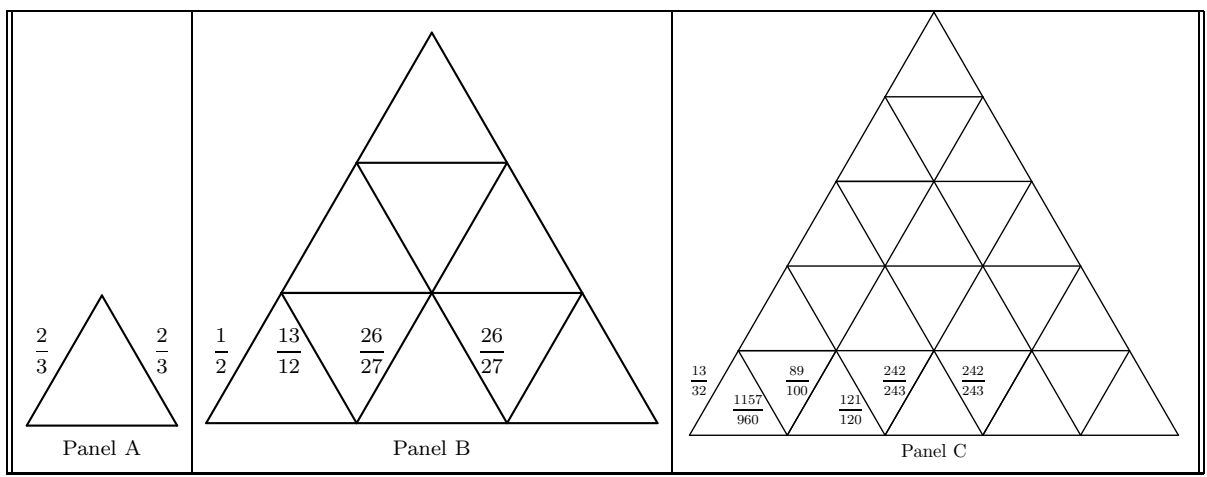

\section{The Circuit Array}

This section formally defines the Circuit array, whose $i$-th row, $0 \le i \le 2(j-2),$ and $j$-th column, $j \ge 1,$ contains $T_{2j-1,j-\lfloor \frac{i+1}{2}\rfloor,LR}^j$
where, as indicated at the end of Section \ref{sec:reduction} the symbol $LR$ means L (respectively R) if $i$ is even (respectively odd).

\begin{example}
This example repeats the derivation of the first three columns  derived at the end of \ref{sec:motivation}. 

Referring to   Figure \ref{fig:motivationillustration}, 
we see Panel A contains row 0, column 1 of the array containing $T_{2j-1,j-i}^j=T_{1,1}^1 = \frac{2}{3}.$ 

Panel B contains column 2 rows 0,1,2, which respectively contain $T_{3,2,L}^2=\frac{26}{27},
T_{3,2,R}^2=\frac{13}{12},
T_{3,1,L}^2=\frac{1}{2},
$

Panel C contains column 3 rows 0,1,2,3,4 which respectively contain
$T_{5,3,L}^3=\frac{242}{243},
T_{5,2,R}^3=\frac{121}{120},
T_{5,2,L}^3=\frac{89}{100},
T_{5,1,R}^3=\frac{1157}{969},
T_{5,1,L}^3=\frac{13}{32}.
$

\end{example}

Table \ref{tab:circuitarrayformal} presents the the first few rows and columns of formal entries of the Circuit Array while 
Table \ref{tab:circuitarray} presents the first few rows and columns of the numerical values of the Circuit Array.

\begin{center}
\begin{table}
 \begin{small}
\caption
{First few rows and columns of the formal entries of  the Circuit Array. }
\label{tab:circuitarrayformal}
{
\renewcommand{\arraystretch}{1.3}
\begin{center}
\begin{tabular}{||c||r|r|r|r|r|r|r|r||} 
\hline \hline

\;&$1$&$2$&$3$&$4$&$5$&$6$&$7$&$\dotsc$\\
\hline
$0$&$T_{1,1,L}^1$&$T_{3,2,L}^2$&$T_{5,3,L}^3$&$T_{7,4,L}^4$&$T_{9,5,L}^5$&$T_{11,6,L}^6$&$T_{13,7,L}^7$&$\dotsc$\\
$1$&\;&$T_{3,2,R}^2$&$T_{5,2,R}^3$&$T_{7,3,R}^4$&$T_{9,4,R}^5$&$T_{11,5,R}^6$&$T_{13,6,R}^7$&$\dotsc$\\
$2$&\;&$T_{3',1,L}^2$&$T_{5,2,L}^3$&$T_{7,3,L}^4$&$T_{9,4,L}^5$&$T_{11,5,L}^6$&$T_{13,6,L}^7$&$\dotsc$\\
$3$&\;&\;&$T_{5,1,R}^3$&$T_{7,2,R}^4$&$T_{9,3,R}^5$&$T_{11,4,R}^6$&$T_{13,5,R}^7$&$\dotsc$\\
$4$&\;&\;&$T_{5,1,L}^3$&$T_{7,2,L}^4$&$T_{9,3,L}^5$&$T_{11,4,L}^6$&$T_{13,5,L}^7$&$\dotsc$\\
$5$&\;&\;&\;&$T_{7,1,R}^4$&$T_{9,2,R}^5$&$T_{11,3,R}^6$&$T_{13,4,R}^7$&$\dotsc$\\
$6$&\;&\;&\;&$T_{7,1,L}^4$&$T_{9,2,L}^5$&$T_{11,3,L}^6$&$T_{13,4,L}^7$&$\dotsc$\\
$7$&\;&\;&\;&\;&$T_{9,1,R}^5$&$T_{11,2,R}^6$&$T_{13,3',R}^7$&$\dotsc$\\
$8$&\;&\;&\;&\;&$T_{9,1,L}^5$&$T_{11,2,L}^6$&$T_{13,3,L}^7$&$\dotsc$\\
$9$&\;&\;&\;&\;&\;&$T_{11,1,R}^6$&$T_{13,2,R}^7$&$\dotsc$\\
$10$&\;&\;&\;&\;&\;&$T_{11,1,L}^6$&$T_{13,2,L}^7$&$\dotsc$\\
$11$&\;&\;&\;&\;&\;&\;&$T_{13,1,R}^7$&$\dotsc$\\
$12$&\;&\;&\;&\;&\;&\;&$T_{13,1,L}^7$&$\dotsc$\\
$13$&\;&\;&\;&\;&\;&\;&\;&$\ddots$\\
$14$&\;&\;&\;&\;&\;&\;&\;&$\ddots$\\

\hline \hline
\end{tabular}
\end{center}
}
 \end{small} 
\end{table}
\end{center}

\begin{center}
\begin{table}
\begin{large}
\caption
{ First few rows and columns of the numerical values of the Circuit Array.}
\label{tab:circuitarray}
{
\renewcommand{\arraystretch}{1.3}
\begin{center}
\begin{tabular}{||c||r|r|r|r|r|r|r||} 
\hline \hline

\;&$1$&$2$&$3$&$4$&$5$&$6$&$\dotsc$\\
\hline \hline
$0$&$\frac{2}{3}$&$\frac{26}{27}$&$\frac{242}{243}$&$\frac{2186}{2187}$&$\frac{19682}{19683}$&$\frac{177146}{177147}$&$\dotsc$\\
$1$&\;&$\frac{13}{12}$&$\frac{121}{120}$&$\frac{1093}{1092}$&$\frac{9841}{9840}$&$\frac{88573}{88572}$&$\dotsc$\\
$2$&\;&$\frac{1}{2}$&$\frac{89}{100}$&$\frac{16243}{16562}$&$\frac{335209}{336200}$&$\frac{108912805}{108958322}$&$\dotsc$\\
$3$&\;&\;&$\frac{1157}{960}$&$\frac{1965403}{1904448}$&$\frac{366383437}{364552320}$&$\frac{ 1071810914005}{1071023961216}$&$\dotsc$\\
$4$&\;&\;&$\frac{13}{32}$&$\frac{305041}{380192}$&$\frac{1303624379}{1372554304}$&$\frac{9044690242835}{9138722473024}$&$\dotsc$\\
$5$&\;&\;&\;&$\frac{224369}{167424}$&$\frac{19373074829}{18067568640}$&$\frac{308084703953915}{303469074613248}$&$\dotsc$\\
$6$&\;&\;&\;&$\frac{89}{256}$&$\frac{296645909}{412902400}$&$\frac{31631261501245}{34990560891392}$&$\dotsc$\\
$7$&\;&\;&\;&\;&$\frac{46041023}{31211520}$&$\frac{112546800611915}{99980909002752}$&$\dotsc$\\
$8$&\;&\;&\;&\;&$\frac{2521}{8192}$&$\frac{320676092095}{495976128512}$&$\dotsc$\\
$9$&\;&\;&\;&\;&\;&$\frac{4910281495}{3059613696}$&$\dotsc$\\
$10$&\;&\;&\;&\;&\;&$\frac{18263}{65536}$&$\dotsc$\\
$\dotsc$&\;&\;&\;&\;&\;&\;&$\dotsc$\\

\hline \hline
\end{tabular}
\end{center}
}
 \end{large}
 \end{table}
\end{center}

The \textit{leftmost} diagonal of the circuit array is defined by 
\begin{equation}\label{equ:leftside}
L_s = T_{2s-1,1,L}^s, s=1,2,3,\dotsc
\end{equation}

\section{The Main Theorem}\label{sec:main}
 
The main theorem asserts that the Circuit Array is a recursive array. Along any fixed row, table values  are a uniform function of previous row and column values. We have already introduced the row 0 function, $G_0,$ \eqref{equ:g0} (see Lemma \ref{lem:row0}) and $G_1(X)$ (see Lemma \ref{lem:row1}). 

\begin{theorem} For each $e \ge 0, \text{ $e$ even,}$ there exist rational functions $G_{e}$ such that for $k \ge 0$  
\begin{equation}\label{equ:maineven}
T_{e+3+2k,2+k,L}^{\frac{e}{2}+2+k} = G_{e}(T_{1+2k,1+k,L}^{1+k}, T_{3+2k,1+k,L}^{2+k}, \dotsc, T_{e+1+2k,1+k,L}^{\frac{e}{2}+1+k}).
\end{equation}
Similarly for each odd, $o =\frac{e}{2}+1$
\begin{equation}\label{equ:mainodd}
T_{\frac{e}{2}+2+2k,1+k,L}^{\frac{e}{2}+2+k} = G_{o}(T_{1+2k,1+k,L}^{1+k}, T_{3+2k,1+k,L}^{2+k}, \dotsc, T_{\frac{e}{2}+2k,1+k,L}^{\frac{e}{2}+1+k}).
\end{equation}
\end{theorem}

Proof of the main theorem is deferred to Sections \ref{sec:s13_proofbasecase} and \ref{sec:proofmain}. Illustrative examples of these recursions are provided in the next section.

\section{Illustrations of the Main Theorem}\label{sec:examples}
 
\begin{example}  For $i=0$ we have by Lemma \ref{lem:row0},
$$
G_0(X)=\frac{X+8}{9}. \text{ Hence, } 
C_{0,2}= \frac{26}{27}=G_0(C_{0,1})= G_0\left(\frac{2}{3}\right), \text{ and }
C_{0,3}= \frac{242}{243}=G_0(C_{0,2})= G_0\left(\frac{26}{27}\right).
$$
Similarly, we have by Lemma \ref{lem:row1},
$$
G_1(X)=\frac{1}{3} \frac{X+8}{X+2}. \text{ Hence, }
C_{1,2}= \frac{13}{12}=G_1(C_{0,1})= G_1\left(\frac{2}{3}\right), \text{ and }
C_{1,3}= \frac{121}{120}=G_1(C_{0,2})= G_1\left(\frac{26}{27}\right).
$$
\end{example}

\begin{example}\label{exa:row2} For $i=1$ we have
$$
G_2(X,Y)=\frac{9Y(X+2)^2+8(X+8)^2}{(X+26)^2}.$$
Hence,
$$C_{2,3}= \frac{89}{100}=G_2(C_{0,1},C_{2,2})= G_2\left(\frac{2}{3},\frac{1}{2}\right), \text{ and }$$
$$C_{2,4}= \frac{16243}{16562}=G_2(C_{0,2},C_{2,3})= G_2\left(\frac{26}{27},\frac{89}{100}\right)
$$
Similarly, we have
$$
G_3(X)=\frac{9Y(X+2)^2 (X+8)+8(X+8)^3}
{9Y(X+2)^2(X+26)+6(X+2)(X+8)(X+26)}.$$
Hence,
$$C_{3,3}= \frac{1157}{960}=G_3(C_{0,1},C_{2,2})= G_3\left(\frac{2}{3},\frac{1}{2}\right), \text{ and }$$
$$C_{3,4}= \frac{1965403}{190448}=G_3(C_{0,2},C_{2,3})= G_3\left(\frac{26}{27},\frac{89}{100}\right).
$$
\end{example}

\begin{example}
For $i=2,$  we have $G_4(X,Y,Z)=\frac{N(X,Y,Z)}{D(X,Y,Z)},$ with
\begin{equation*}
N(X,Y,Z) =
\begin{cases}
512(X+2)^0(X+8)^5(X+80)Y^0+\\
1152(X+2)^2(X+8)^3(X+80)Y^1+ \\
648(X+2)^4(X+8)^1(X+80)Y^2+ \\
36(X+2)^2(X+8)^2(X+80)^2 Y^0 Z+ \\
108(X+2)^3(X+8)^1(X+80)^2Y^1 Z+ \\
81(X+2)^4(X+8)^0(X+80)^2Y^2Z,
\end{cases}
\end{equation*}
and
\begin{equation*}
D(X,Y,Z)=
\begin{cases}
 676(X+2)^0(X+8)^2Q(X)^2Y^0 +\\
1404(X+2)^2(X+8)^2Q(X)^1Y^1+ \\
729(X+2)^4(X+8)^2Q(X)^0Y^2,
\end{cases}
\end{equation*}
with, $Q(X)=13X^2+298X+2848.$

These polynomials are formatted to show certain underlying patterns the statement and proof of which will be the subject of another paper.

One then has $C_{4,4}=\frac{305041}{380192}=
G_4(C_{0,1}, C_{2,2}, C_{4,3})=
G_4\left(\frac{2}{3}, \frac{1}{2}, \frac{13}{32}\right).$
\end{example}

\section{Alternate Approaches to the Main Theorem}

The main theorem formulates the recursiveness of the circuit array in terms of recursions by rows with the number of arguments of these recursions growing by row. There are other approaches to formulating the main theorem, explored in the next few sections.  
\begin{itemize}
\item Section \ref{sec:closed} explores a formulation of the main theorem in terms of closed formula similar to those found in Section \ref{sec:motivation}
\item Section \ref{sec:closed} also explores formulation of the main theorem in terms of a single variable rather than multiple variables.
\item Section \ref{sec:determinant} explores determining an LHRCC for the numerators of the leftmost diagonal and strongly conjectures its impossibility. This contrasts with other 2-dimensional arrays whose diagonals do satisfy LHRCC.
\item Section \ref{sec:product} explores asymptotic approximations to the leftmost diagonal. 
\end{itemize}

\section{Recursions vs. Closed Formulae}\label{sec:closed}

This section explores a closed-formula approach to the main theorem. We begin with a review.  

We have already seen  (Lemmas \ref{lem:row0} and \ref{lem:row1}) that row 0 of the circuit array has a simple closed form,
$$
		C_{0,s} = 1 - \frac{3}{9^s} \qquad s \ge 1;
$$
and similarly, row 1 also has a simple closed form,
$$
		C_{1,s} = 1+\frac{2}{3} \frac{1}{9^{s-1}-1}.
$$

This naturally motivated seeking a formulation of the entire array in terms of closed formulae. However, this approach quickly becomes excessively cumbersome.  For example, consider row 2. With the aid of  \cite[A163102,A191008]{OEIS},  we found the following closed form for this row:
\begin{multline*}
\textbf{Define } n=2(s-2), \qquad d=\frac{1}{2}\biggl( 3^{s-1}-1 \biggr)\\
N=\frac{1}{4} \biggl(n\cdot3^{n+1}\biggr)+
\frac{1}{16} \biggl(5\cdot3^{n+1}+(-1)^n\biggr),
\qquad
D=\frac{1}{2} d^2 (d+1)^2
\end{multline*}
then $$ C_{2,s} = 1 - \frac{N}{D}.$$

However, this formula is much more complicated than the formula presented in Section \ref{sec:examples},
$$
C_{2,s} = G_2(X,Y)=\frac{9Y(X+2)^2+8(X+8)^2}{(X+26)^2}, \qquad \text{ with } X=C_{0,s-2}, Y= C_{2,s-1}.
$$

We present one more attempt at a closed formula which also failed, that of using a single variable.  We begin by first re-labeling $\frac{2}{3}$ as $1-\frac{3}{x}$ in the first reduction of an all--one $n$-grid (see Panel A in Figure \ref{fig:motivationillustration}).  If we then continue reductions, all labels are rational functions in this single variable $x$ so that upon substitution of $x=9$ we may then obtain desired resistance edge labels.
  
  As before, the resulting formulas are highly complex. We present below these closed formulas for the left-side diagonal, $L_s$. They are derived by ``plugging in" to the four basic transformation functions of Section \ref{sec:proofmethods} as we did in Section \ref{sec:motivation}.

\begin{itemize}

\item
 \[\ \frac{x-3}{x-1}
\qquad \text{gives $L_1=\frac{2}{3}$ when $x=9$}\]

\item
\[\frac{2}{3}\frac{x-3}{x-1}
\qquad \text{gives $L_2=\frac{1}{2}$ when $x=9$}\]

\item
\[\frac{(x-3)(3x-1)}{6(x-1)^2}, \text{gives $L_3=\frac{13}{32}$ when $x=9$}\]

\item
\[\frac{(x-3)(3(x-1)(x-3) + 4(3x-1)^2)}{96(x-1)^3}, \text{gives $L_4=\frac{89}{256}$ when $x=9$}\]

\item
\[\frac{(x-3)(3(x-1)(x-3)(34x-18) + 16(3x-1)^3)}{1536(x-1)^4}, \qquad
\text{gives $L_5$ when $x=9$}
\]

\item
\[\frac{(x-3)(3(x-1)(x-3)(793x^2-874x+273) + 64(3x-1)^4)}{24576(x-1)^5}, \qquad
\text{gives $L_6$ when $x=9$}
\]

\item
\[\frac{(x-3) (6(x - 1)(x - 3)(7895x^3 - 13549x^2 + 8693x - 2015)+4^4(3x-1)^5)}{393216 (x-1)^6}, \qquad
\text{gives $L_7$ when $x=9$}.
\]

\end{itemize}

There are interesting patterns in the above results and it may yield future results. One example of an interesting pattern is found in the constants appearing in the denominators. For $s \ge 3$ the denominator constants in the formulas yielding $L_s$ upon substitution of $x=9,$ satisfy $3\times2^{4(s-3)+1}$ We however do not further pursue this in this paper.

To sum up, because of the greater complexity as well as lack of completely describable patterns in the closed formula we abandoned this approach in favor of a recursive approach in several variables.

\section{Impossibility of a recursive sequence for the left-most diagonal}\label{sec:determinant}

It is natural, when studying sequences of fractions, to separately study their numerators and denominators. We have seen that for $C_0, C_1$ such an approach uncovers  LHRCC.   Therefore, it comes as a surprise to have a result stating the impossibility of an LHRCC.

To present this impossibility result, we first, briefly review a technique for discovering LHRCC. Suppose we have an integer sequence such as $G_1, G_2, \dotsc$ Suppose further we believe this sequence is second order, that is,
$$
    G_n = x G_{n-1}+y G_{n-2}
$$
As $n$ varies this last equation generates an infinite number of equations in $x$ and $y.$ 
In other words, to investigate the possible recursiveness of this sequence we can solve the following set of equations for any $m$ and the use the solution to test further,
\begin{center}
$\begin{bmatrix} G_m & G_{m+1} \\ G_{m+1} & G_{m+2} \end{bmatrix}$
$ \begin{bmatrix} x \\ y \end{bmatrix}$ 
$=$
$ \begin{bmatrix} G_{m+2} & G_{m+3} \end{bmatrix}.$
\end{center}
Solving this set of equations by Cramer's rule naturally motivates considering the determinant
$$  \begin{vmatrix} G_m & G_{m+1} \\ G_{m+1} & G_{m+2} \end{vmatrix} $$
for any integer $m.$ While these determinants are non-zero, the order 3 determinants,
$$  \begin{vmatrix} G_m & G_{m+1} & G_{m+2}
  \\ G_{m+1} & G_{m+2} & G_{m+3} \\
  G_{m+2} & G_{m+3} & G_{m+4} 
  \end{vmatrix}, $$
 must be zero because of the dependency captured by the LHRCC.

These remarks generalize to $r$-th order recursions for integer $r \ge 2,$ and explain why in the search for recursions it is natural to consider such determinants. It follows that if for some $m$ and for all $r \ge 2$ the following determinant is non-zero, 
$$  \begin{vmatrix} G_m & G_{m+1} & \dotsc & G_{m+r}
  \\ 
  G_{m+1} & G_{m+2} & \dotsc & G_{m+r+1} \\
  \dotsc  & \dotsc  & \dotsc & \dotsc \\
  G_{m+r-1} & G_{m+r} &  \dotsc & G_{m+2r-1} 
  \end{vmatrix}, $$
then it is impossible for the sequence $\{G_m\}$ to satisfy any LHRCC.

The following conjecture, verified for several dozen early values of $k$ shows a remarkable and unexpected simplicity in the values of these determinants.  
 
\begin{conjecture}
Let $T(j)= \frac{j(j+1)}{2}$ indicate the $j$-th triangular number.  
Using \eqref{equ:leftside}, define ${n'}_s$ and ${d'}_s$ by
$L_s = \frac{n_s}{d_s}=\frac{n'_s}{2^{4s-7}}$, where $n_s$ and $d_s$ are relatively prime.   For any $j \ge 2$ we have
\begin{center}
$$\begin{vmatrix} n'_2 & n'_3 & \cdots& n'_{2+j}\\
n'_3 & n'_4 & \cdots &n_{3+j}\\
\vdots & \vdots & \ddots &\vdots\\
n'_{2+j}& n'_{3+j} & \cdots & n'_{2+2j}\end{vmatrix}
=9^{T(j-1)}.$$
\end{center}
\end{conjecture}
 
 \begin{corollary} Under the conditions stated in the conjectures, it is impossible for the  $\{{n'}_s\}_{s \ge 1}$ to satisfy an LHRCC of any order.
 \end{corollary}
 
 \begin{comment}
It is tempting to suggest that the numerators satisfy no LHRCC because they are growing too fast. But that is not true. We know that $L_s <1,$
\cite[Corollary 7.2]{EvansHendel} and that the denominators form a geometric sequence.  It follows that the numerators are bounded by a geometric sequence. In terms of growth rate, there is no reason why the sequence shouldn't be able to satisfy an LHRCC.
 \end{comment}


\section{An Asymptotic Approach}\label{sec:product}

Prior to presenting the proof of the main theorem, we explore one more approach in this section. By way of motivation recall that several infinite arrays have asymptotic formulas associated with them.  For example, the central binomial coefficients have asymptotic formulas arising from  Stirling's formula.

For purposes of expositional smoothness, we focus on the leftmost diagonal,  
$L_s,$ \eqref{equ:leftside}.

 Hendel, \cite{Hendel} introduced the idea of finding explicit formulas for edge-values in terms of products of factors. After 
numerical experimentation, the following approximation was found,  
\begin{equation}\label{equ:A}
	L_s \asymp A_s = \frac{2}{3} \displaystyle \prod_{i=2}^s (1 - \frac{1}{2i-1}),
\end{equation}
with $A$ standing for approximation. Tables \ref{tab:leftcenter5rows} and \ref{tab:leftcenter80rows} provide numerical evidence for this approximation. The key takeaways from both tables is that both differences $L_s - A_s$ and ratios $\frac{L_s}{A_s}$ are monotone decreasing for $s \ge 3.$

\begin{center}
\begin{table}
 \begin{small}
\caption
{Numerical evidence for conjectures about $L_s,$ first five rows. Notice that after $s=2$ all difference and ratio columns are monotone decreasing.}
\label{tab:leftcenter5rows}
{
\renewcommand{\arraystretch}{1.3}
\begin{center}
\begin{tabular}{||c||c|c|c|c||c|c|c||c|c||} 
\hline \hline

$s$&$L_s$&$A_s$&$L_s-A_s$&$\frac{L_s}{A_s}$&$P_s$&$A_s-P_s$&$\frac{A_s}{P_s}$&$L_s-P_s$&$\frac{L_s}{P_s}$\\
\hline
$1$&$0.6667$&$0.6667$&$0$&$1$&$0.5908$&$0.0758$&$1.1284$&$0.0758$&$1.1284$\\
$2$&$0.5$&$0.4444$&$0.0556$&$1.125$&$0.4178$&$0.0267$&$1.0638$&$0.0822$&$1.1968$\\
$3$&$0.4063$&$0.3556$&$0.0507$&$1.1426$&$0.3411$&$0.0144$&$1.0424$&$0.0651$&$1.191$\\
$4$&$0.3477$&$0.3048$&$0.0429$&$1.1407$&$0.2954$&$0.0094$&$1.0317$&$0.0522$&$1.1769$\\
$5$&$0.3077$&$0.2709$&$0.0368$&$1.136$&$0.2642$&$0.0067$&$1.0253$&$0.0435$&$1.1647$\\

\hline \hline
\end{tabular}
\end{center}
}
 \end{small} 
\end{table}
\end{center}

\begin{center}
\begin{table}
 \begin{small}
\caption
{Numerical evidence for conjectures about $L_s,$ first 80 rows. Observe that except for a few initial values the difference and ratio columns are monotone decreasing.}
\label{tab:leftcenter80rows}
{
\renewcommand{\arraystretch}{1.3}
\begin{center}
\begin{tabular}{||c||c|c|c|c||c|c|c||c|c||} 
\hline \hline

 $s$&$L_s$&$A_s$&$L_s-A_s$&$L_s/A_s$&$P_s$&$A_s-P_s$&$A_s/P_s$&$L_s-P_s$&$L_s/P_s$\\
\hline
$8$&$0.2387$&$0.2122$&$0.0265$&$1.125$&$0.2089$&$0.0033$&$1.0157$&$0.0298$&$1.1427$\\
$16$&$0.1658$&$0.1489$&$0.017$&$1.1141$&$0.1477$&$0.0012$&$1.0078$&$0.0181$&$1.1228$\\
$24$&$0.1346$&$0.1212$&$0.0134$&$1.1103$&$0.1206$&$0.0006$&$1.0052$&$0.014$&$1.1161$\\
$32$&$0.1162$&$0.1049$&$0.0114$&$1.1084$&$0.1044$&$0.0004$&$1.0039$&$0.0118$&$1.1127$\\
$40$&$0.1038$&$0.0937$&$0.0101$&$1.1072$&$0.0934$&$0.0003$&$1.0031$&$0.0103$&$1.1107$\\
$48$&$0.0946$&$0.0855$&$0.0091$&$1.1065$&$0.0853$&$0.0002$&$1.0026$&$0.0093$&$1.1094$\\
$56$&$0.0875$&$0.0791$&$0.0084$&$1.1059$&$0.079$&$0.0002$&$1.0022$&$0.0086$&$1.1084$\\
$64$&$0.0818$&$0.074$&$0.0078$&$1.1055$&$0.0739$&$0.0001$&$1.002$&$0.008$&$1.1077$\\
$72$&$0.0771$&$0.0697$&$0.0073$&$1.1052$&$0.0696$&$0.0001$&$1.0017$&$0.0075$&$1.1071$\\
$80$&$0.0731$&$0.0662$&$0.0069$&$1.105$&$0.0661$&$0.0001$&$1.0016$&$0.007$&$1.1067$\\

\hline \hline
\end{tabular}
\end{center}
}
 \end{small} 
\end{table}
\end{center}

The $P$ columns in these tables (which also provide good approximations as measured by differences and ratios) correspond to the following further approximation
\begin{equation}\label{equ:P}
	A_s \asymp P_s = \sqrt{\frac{\pi}{9s}}, 
\end{equation}
with $P$ standing for the approximation of $A_s$ with $\pi$.   \color{black}

Equation \eqref{equ:P} is naturally derived from \eqref{equ:A} using Stirling's formula. The next lemma contains a formal statement of the result.

\begin{lemma}
$$
	A_s \asymp P_s.
$$
\end{lemma}
\begin{proof}
By \eqref{equ:A} we have
$$
	A_s = \frac{2}{3} \Biggl( \frac{2}{3} \frac{4}{5} \dotsc \frac{2s-2}{2s-1} \Biggr) .
$$
Applying the identity  $(2s-1)! = \Biggl(2 \cdot 4 \cdot \dotsc \cdot 2s-2 \Biggr) \Biggl(  3 \cdot 5 \cdot \dotsc 2s-1\Biggr)$ to the last equation, we have
 
$$
	A_s =  \frac{2}{3} \frac{\Biggl( 2^{s-1} (s-1)! \Biggr)^2} {(2s-1)!}.  
$$
Of the many forms of Stirling's formula, we can simplify the last equation by applying the standard approximation (see for example~\cite{Wolfram}) $n! \asymp \Biggl( \frac{n}{e} \Biggr)^n \sqrt{2\pi n} $, yielding

$$
	A_s = \frac{2}{3} \frac{4^s}{4} \Biggl( \frac{s-1}{e} \Biggr)^{2(s-1)} \biggl(2\pi(s-1) \biggr)  \Biggl(\frac{e}{2s-1}\Biggr)^{2s-1} \frac{1}{\sqrt{2 \pi (2s-1)}}.
$$

By gathering constants, cancelling the powers of $e,$ and using the fact that $c_1 s -c_2 \asymp s$ for constants $c_1, c_2,$ we can simplify this last equation to 
$$
	A_s = \frac{e}{6} 4^s  \sqrt{s} \sqrt{\pi} (s-1)^{2s-2} \left(\frac{1}{2s-1}\right)^{2s-1}.
$$
Further simplification is obtained by using traditional calculus identities on limits resulting in powers of $e.$
$$
	(s-1)^{2s-2} = \Biggl( \frac{s-1}{s} \Biggr)^{2s} s^{2s} \frac{1}{(s-1)^2}  \asymp e^{-2} s^{2s} \frac{1}{s^2},$$ $$
	\frac{1}{(2s-1)^{2s-1}} = \Biggl(\frac{2s}{2s-1}\Biggr)^{2s-1} \frac{1}{(2s)^{2s-1}} \asymp e \frac{1}{4^s} \frac{1}{s^{2s}}2s.
$$
Combining these last 3 equations, cancelling powers of $e$ and $4,$  and using the fact that $c_1 s + c_2 \asymp s,$ we obtain
$$
	A_s \asymp \sqrt{\pi} \frac{1}{6}  2s \frac{1}{s^2} \sqrt{s} = \frac{\sqrt{\pi}}{3 \sqrt{s}} = \sqrt{\frac{\pi}{9s}} = P_s
$$
as required.
\end{proof}
\color{black}


\section{Base Case of the Inductive Proof}\label{sec:s13_proofbasecase}

The proof of the main theorem is by induction on the row index, parametrized by whether the row is even or odd, as shown in equations \eqref{equ:maineven}-\eqref{equ:mainodd}. The base case requires proofs for rows 0,1,2,3. 

We suffice throughout the proof with consideration of the  the even rows, the proof for the odd rows being highly similar and hence omitted. The proof for row 0 has already been completed and is summarized in Lemma \ref{lem:row0}. Recall that the proof was based on the equations describing a non-boundary left edge (\eqref{equ:leftside3proofs} and \eqref{equ:leftside9proofs}) as well as Theorem \ref{the:uniformcenter}. A proof for rows 0 and one can be found in Lemmas \ref{lem:row0} and \ref{lem:row1}.  Proofs in this and the next section are accomplished similarly by applying the appropriate transformation functions found in Section~\ref{sec:proofmethods} as well as the Uniform Center Theorem, Theorem \ref{the:uniformcenter}.

In this section we show \eqref{equ:maineven} for the case $e=2.$  We accomplish this by first proving \eqref{equ:maineven} when $k=0$  and then  proving for $k>0.$ This separation into two cases is for expositional clarity since the proof can be accomplished with the single arbitrary case.

\noindent\textsc{Case 1: Proof of Equation \eqref{equ:maineven} for $e=2,k=0.$} 

We must show
\begin{equation}\label{equ:tempbasecase1}
    T_{5,2,L}^3 = G_3(T_{1,1,L}^1, T_{3,1,L}^2),
\end{equation}
for some function $G_3.$

By the formula for non boundary left edges, \eqref{equ:leftside3proofs}, we know 
\begin{equation}\label{equ:tempbasecase2}
    T_{5,2,L}^3 = F(T_{5,1}^2,T_{5,2}^2, T_{6,2}^2).
\end{equation}

Proceeding as in the proof of Lemma \ref{lem:row0}  we have as follows:
\begin{itemize}
    \item By Theorem \ref{the:uniformcenter}(b) the six edges of triangles $T_{5,2}^2, T_{6,2}^2$ are identically one.
    \item By \eqref{equ:uniformcenter}, the uniform center for the first diagonal in the twice reduced $n$-grid begins on row $s+d=2+1=3.$
    Therefore, the argument $T_{5,1}^2$ in \eqref{equ:tempbasecase2} may be replaced by the identically labeled triangle $T_{3,1}^2.$
    \item Triangle  $T_{3,1}^2$ has three sides, $T_{3,1,L}^2, T_{3,2,R}^2, T_{3,2,B}^2.$ 
    \item But by Lemma \ref{lem:row1},  
    $T_{3,2,R}^2= G_0(T_{1,1,L}^1),$ and by Theorem \ref{the:uniformcenter}(c), 
    $T_{3,2,B}^2=T_{3,2,R}^2$
\end{itemize}

Applying the above to \eqref{equ:tempbasecase2} and plugging into \eqref{equ:leftside9proofs} we have
$$
T_{5,2,L}^3 = Y(\Delta(T_{3,1,L}^2,G_0(T_{1,1,L}^1), G_0(T_{1,1,L}^1)), \Delta(1,1,1),\Delta(1,1,1))
=G_3(T_{1,1,L}^1, T_{3,1,L}^2),
$$
which has the required form of \eqref{equ:tempbasecase1} as desired. This completes the proof of \eqref{equ:maineven} for the case $e=2,k=0.$

\noindent\textsc{Case 2: Proof of Equation \eqref{equ:maineven} for $e=2,k>0.$}

Proceeding exactly as we did in the case $k=0$ we have by \eqref{equ:leftside3proofs},
\begin{equation}
\label{equ:tempbasecase3}
    T_{5+2K,2+K,L}^{3+K} = F(T_{5+2K,1+K}^{2+K},T_{5+2K,2+K}^{2+K}, T_{6+2K,2+K}^{2+K}).
\end{equation}

Continuing as in the case $k=0$ we have:
\begin{itemize}
    \item By Theorem \ref{the:uniformcenter}(b) the six edges of triangles $T_{5+2K,2+K}^{2+K}, T_{6+2K,2+K}^{2+K}$ are identically one.
    \item By \eqref{equ:uniformcenter}, the uniform center for (1+K)th diagonal of the all-one $n$ grid reduced $2+K$ times begins on row $s+d=2+K+1+K=3+2K.$
    Therefore, the argument $T_{5+2K,1+K}^{2+K}$ in \eqref{equ:tempbasecase3} may be replaced by the identically labeled triangle $T_{3+2K,1+K}^{2+K}.$
    \item Triangle $T_{3+2K,1+K}^{2+K}$ has three sides, $T_{3+2K,1+K,L}^{2+K}, T_{3+2K,1+K,R}^{2+K}, T_{3+2K,2+K,B}^{2+K}.$ 
    \item But by Lemma \ref{lem:row1},  
    $T_{3+2K,2+K,R}^{2+K}= G_0(T_{1+2K,1+K,L}^{1+K}),$ and by Theorem \ref{the:uniformcenter}(c), 
    $T_{3+2K,2+K,B}^{2+K}=T_{3+2K,2+K,R}^{2+K}$
\end{itemize}

Applying the above to \eqref{equ:tempbasecase3} and plugging into the edge version of the equation, \eqref{equ:leftside9proofs}, we have
\begin{multline*}
T_{5+2K,2+K,L}^{3+K} = Y(\Delta(T_{3+2K,1+K,L}^{2+K},G_0(T_{1+2K,1+K,L}^{1+K}), G_0(T_{1+2K,1+K,L}^{1+K})), \Delta(1,1,1),\Delta(1,1,1))\\
=G_3(T_{1+2K,1+K,L}^{1+K}, T_{3+2K,1+K,L}^{2+K}),
\end{multline*}
which has the required form of \eqref{equ:tempbasecase1} as was to be shown. This completes the proof of \eqref{equ:maineven} for the second case and hence completes the proof of the base case $e=2.$


 \section{Proof of the Main Theorem}\label{sec:proofmain}

This section completes the inductive proof of the main theorem, by showing equations \eqref{equ:maineven} and \eqref{equ:mainodd}, the base case of which was completed in the prior section. Accordingly throughout this section we assume $E$ an even number, corresponding to row $E$ of the Circuit Array, such that,
\begin{equation}\label{equ:ebigger4}
    E \ge 4.
\end{equation}

We will utilize the following lemma, whose proof follows from an inspection of Table \ref{tab:circuitarrayformal}.

\begin{lemma}\label{lem:isamember}
Triangle $T^a_{b,c,LR}$ belongs to row $d$ column $e$ of the Circuit array if $b=2a-1,$ $a=e$ and either i) $LR=L, a=c, d=0,$ ii) $LR=L, d=2(a-c)$, or iii) $LR=R, d=2(a-c)-1>0.$
\end{lemma}

For an induction assumption we assume \eqref{equ:maineven} and \eqref{equ:mainodd} hold for all $e < E$ and proceed to prove these equations for the case $E.$ We suffice with the proof for even rows (i.e, Equation \eqref{equ:maineven}) the proof for odd rows being similar and hence omitted. The proof proceeds in a manner similar the proofs presented in the prior section.

First, by \eqref{equ:leftside3proofs}
we have
\begin{equation}\label{equ:tempfirst}
    T^{\frac{E}{2}+2+k}_{E+3+2k, 2+k, L}=
    F(T^{\frac{E}{2}+1+k}_{E+3+2k, 1+k},
    T^{\frac{E}{2}+1+k}_{E+3+2k, 2+k},
    T^{\frac{E}{2}+1+k}_{E+4+2k, 2+k})
\end{equation}

Second, utilizing assumption \eqref{equ:ebigger4} and using part (a) of the Uniform Center Theorem, in the three triangle arguments on the right hand side of \eqref{equ:tempfirst}  $E+3+2k$ and $E+4+2k$ can be replaced with $E+1+2k$ since
$$
s+d =\frac{E}{2}+1+k+1+k \ge E+1+k; \quad
\text{ and similarly }
s+d = \frac{E}{2}+1+k+2+k \ge E+1+k,$$
implying
\begin{equation}\label{equ:tempsecond}
    T^{\frac{E}{2}+2+k}_{E+3+2k, 2+k, L}=
    F(T^{\frac{E}{2}+1+k}_{E+1+2k, 1+k},
    T^{\frac{E}{2}+1+k}_{E+1+2k, 2+k},
    T^{\frac{E}{2}+1+k}_{E+1+2k, 2+k}).
\end{equation}

Third, therefore  expanding \eqref{equ:leftside3proofs} to a full 9 variable function, \eqref{equ:leftside9proofs}, and using Theorem \ref{the:uniformcenter}(c) we have that \eqref{equ:tempsecond} is expanded to
\begin{multline*}
    T^{\frac{E}{2}+2+k}_{E+3+2k, 1+k, L}=
     Y(\Delta(T^{\frac{E}{2}+1+k}_{E+1+2k,1+k,L},
     T^{\frac{E}{2}+1+k}_{E+1+2k,1+k,R},
      T^{\frac{E}{2}+1+k}_{E+1+2k,1+k,R}),\\
      \Delta(T^{\frac{E}{2}+1+k}_{E+1+2k,2+k,L},
      T^{\frac{E}{2}+1+k}_{E+1+2k,2+k,R},
      T^{\frac{E}{2}+1+k}_{E+1+2k,2+k,R}),\\
      \Delta(T^{\frac{E}{2}+1+k}_{E+1+2k,2+k,L},
      T^{\frac{E}{2}+1+k}_{E+1+2k,2+k,R},
      T^{\frac{E}{2}+1+k}_{E+1+2k,2+k,R})).
\end{multline*}

Fourth, in examining the rows to which the arguments of this last equation belong using Lemma \ref{lem:isamember} we have  
\begin{itemize}
    \item $T^{\frac{E}{2}+1+k}_{E+1+2k, 1+k, L}$ in row $E$,
    \item $T^{\frac{E}{2}+1+k}_{E+1+2k, 1+k, R}$ in row $E-1$,
    \item $T^{\frac{E}{2}+1+k}_{E+1+2k, 2+k, L}$ in row $E-2$, and
    \item $T^{\frac{E}{2}+1+k}_{E+1+2k, 2+k, R}$ in row $E-3$.
\end{itemize}

The proof is completed by the induction assumption applied to rows $E-1, E-2, E-3.$ More specifically we must show that when $k=0$ the second row element is a function of all leftmost diagonal elements. But the first argument on the right hand side  is the leftmost diagonal element of row $E$ while the induction assumption assures us that the other arguments are functions of the leftmost elements of previous rows, $e < E.$ This completes the proof.


\end{document}